\documentclass[11pt,a4]{amsart}

\usepackage{mathrsfs}
\usepackage[colorlinks,citecolor=red,pagebackref,hypertexnames=false, breaklinks]{hyperref}

\usepackage{amsmath, color}
\usepackage{enumerate}
\usepackage{amssymb}
\usepackage{amsfonts}

\newcommand{\disappear}[1]

\newtheorem{theorem}{Theorem}[section]
\newtheorem{proposition}[theorem]{Proposition}

\newtheorem{lemma}[theorem]{Lemma}

\theoremstyle{remark} \newtheorem{remark}[theorem]{Remark}
\newtheorem{definition}[theorem]{Definition}

\newcommand\R{\mathbb{R}}

\newcommand\T{\mathbf{T}}
\newcommand\M{\mathcal{M}}

\newcommand\Z{\mathbf{Z}}

\def\d{{d}}

\newcommand{\supp}{\operatorname{supp}}
\newcommand{\vol}{\operatorname{vol}}
\newcommand\Id{\operatorname{Id}}
\newcommand\ang[1]{\langle #1 \rangle}

\begin{document}

\title[Riesz transforms on a class of non-doubling manifolds]{Riesz transforms on a class of non-doubling manifolds}
\author{Andrew Hassell and Adam Sikora}
\address{Andrew Hassell, Department of Mathematics, Australian National University,
ACT 0200 Australia
}
\email{Andrew.Hassell@anu.edu.au}
\address {Adam Sikora, Department of Mathematics, Macquarie University, NSW 2109, Australia}
\email{Adam.Sikora@mq.edu.au}

\subjclass{42B20 (primary), 47F05, 58J05 (secondary).}
\keywords{Riesz transform, heat kernel bounds, resolvent estimates, non-doubling spaces}

\begin{abstract}
We consider a class of manifolds $\mathcal{M}$ obtained by taking the connected sum of a finite number of $N$-dimensional Riemannian manifolds of the form $(\mathbb{R}^{n_i}, \delta) \times (\mathcal{M}_i, g)$, 
where $\mathcal{M}_i$ is a compact manifold, with the product metric. The case of greatest interest is when the Euclidean dimensions $n_i$ are not all equal. This means that
the ends have different `asymptotic dimension', and implies that the Riemannian manifold $\mathcal{M}$ is not a doubling space. 
We completely describe the range of exponents $p$ for which the Riesz transform on $\mathcal{M}$ is a bounded operator on $L^p(\mathcal{M})$. Namely, under the assumption that each $n_i$ is at least $3$, we show that Riesz transform is of weak type $(1,1)$, 
is continuous on $L^p$ for all $p \in (1, \min_i n_i)$, and is unbounded on $L^p$ otherwise. This generalizes results of the first-named author with Carron and Coulhon devoted to the doubling case of the connected sum of several copies of Euclidean space  $\mathbb{R}^{N}$, and of Carron concerning the Riesz transform on connected sums. 
\end{abstract}

\maketitle

\section{Introduction}\label{sone}

A central topic in harmonic analysis and linear partial differential equations is $L^p$ boundedness of the Riesz transform. Riesz transforms have been studied for
almost 100 years, starting with the classical work of Riesz \cite{Ri} on $L^p$ boundedness of the Hilbert transform. 
In the setting of Riemannian manifolds, the question  can just be formulated in the following way. 
 Let $\mathcal{M}$ be a complete Riemannian manifold, let $\Delta$ denote the positive Laplacian on $\mathcal{M}$ and $\nabla$ the gradient corresponding to the Riemannian structure. The Riesz transform operator on $\mathcal{M}$ is then defined by 
$$
T = \nabla \Delta^{-1/2}.
$$
The goal is to determine the range of $p \in [1, \infty]$ for which $T$ acts as a bounded operator from $L^p(\mathcal{M}) \to L^p(\mathcal{M}; T\mathcal{M})$. In addition, in the case of $p=1$,  the question of a weak-type $(1,1)$ estimate for $T$
is significant and the Riesz transform is also studied in the context of BMO and Hardy space theory. In fact in some settings the operator $T$ can be used to define these functional spaces. 

To our knowledge the result that we present in this article is the first example of a setting in which
the doubling condition, see \eqref{doub} below, fails and still the range of $L^p$ spaces on which 
the Riesz transform is bounded is fully described. This is the main motivation for our study. 
To describe the method in the simplest possible setting we work in a very specific class of manifolds, namely, connected sums of products of Euclidean spaces and compact manifolds (already studied in the heat kernel literature, see for example \cite{GS99, GS, GS2}). However, we show in Section~\ref{gen} that our method applies much more generally to connected sums of manifolds. 

In the setting of Riemannian manifolds the question of continuity  of $T$ was formulated by Strichartz \cite{Strichartz} in 1983.  
An important motivation for the investigation of the Riesz transform comes from the fact that the boundedness of the Riesz transform means that the first order $L^p$-based Sobolev space can be defined using the functional calculus of the Laplacian, rather than a full set of (usually non-commuting) vector fields, which is very convenient and useful from the standpoint of operator theory, heat kernel estimates, and so on.  

The literature is too vast to summarize here, but we mention a few seminal works \cite{Strichartz}, \cite{Ba2}, \cite{Lohoue}, \cite{CD}, \cite{ACDH}. See the introduction of \cite{ACDH} for a more detailed literature discussion.  The research in the area of $L^p$ boundedness of the Riesz transform is today more active than ever. A small sample of recent studies in this field can be found for example in  
\cite{Dev, Dev2, Ji1, Ji2} and references therein.
A good survey of the existing results and further references can be also found in \cite{Au}. 

\vskip 8pt

We consider an $N$-dimensional complete Riemannian manifold $ \mathcal{M}$ that is formed by taking the connected sum of $l \geq 2$
copies of manifolds which are products of Euclidean spaces $\R^{n_i}$ with compact boundaryless Riemannian manifolds $\mathcal{M}_i$. 
Thus the manifold consists of the union of a compact part, say $K$, and $l$ ``ends", which are products of Euclidean spaces and compact spaces, and we assume that the Riemannian metric on each end is the product metric. 
Of course, we must have $\dim \mathcal{M}_i + n_i = N$, for each~$i$. In what follows we always assume that each $n_i$ is at least $3$.  Our aim is to  determine the range of $p \in [1, \infty]$ for which the Riesz transform  $T=\nabla \Delta^{-1/2}$ acts as a bounded operator from $L^p(\mathcal{M}) \to L^p(\mathcal{M}; T\mathcal{M})$. We also establish weak-type $(1,1)$ estimates for $T$ in this class of manifolds.

One extreme case is where each $\mathcal{M}_i$ is just a point; that is, we have a connected sum of Euclidean spaces. This case was treated in \cite{CCH}. Various generalisations and extensions of this result were studied in 
\cite{Ca1, Ca2} and \cite{Ji1}. Our interest here is in cases where the dimensions $n_i$ of the Euclidean spaces are not all the same. Since the asymptotic dimension in the sense of Gromov of each end is $n_i$, independent of the compact factor $\mathcal{M}_i$, we can think of this intuitively as a ``connected sum of Euclidean spaces of different dimensions". Such a class of manifolds was first studied by Grigor'yan and Saloff-Coste, who obtained upper and lower bounds on the heat kernel on such manifolds, see \cite{GS99, GS}. 
Other aspects of harmonic analysis on such spaces are being investigated by Bui, Duong, Li and Wick \cite{BDLW}. 

More generally, one can consider connected sums of general classes of manifolds. In \cite[Open Problem 8.2]{CCH}, it was asked what conditions are required such that if the Riesz transform is bounded on $L^p$ on several spaces, then it is bounded on the connected sum. This question was partially answered by G. Carron in \cite{Ca2} who proved Proposition~\ref{prop:Carron} below --- see the discussion below Theorem~\ref{thm:main} for more information. Our main theorem improves on Carron's result by determining the optimal range of $p$ and also including the case $n_i = 3$. B. Devyver obtained other sufficient conditions on the boundedness of the Riesz transform on connected sums in \cite{Dev}. 

From the point of view of harmonic analysis, the key feature of this class of manifolds is that \emph{the Riemannian measure does not satisfy the doubling property}.
Let us recall that a metric measured space  $(\mathcal{X}, d, \mu)$ with metric $d$ and Borel measure $\mu$ is said to satisfy the \emph{doubling 
condition} that is if there exists universal constant $C$ such that 
\begin{eqnarray}
\mu(B(x,2r))\leq C \mu(B(x,r))\quad \forall\,r>0,\,x\in X.  \label{doub}
\end{eqnarray}
Here by $B(x,r)$ we denote the ball of radius $r$ centred at $x\in X$.

 Now, suppose that $n_i$ is strictly less than $n_j$. Then a large ball of radius $R$ contained in the $i$th end will have measure approximately $c_{n_i} R^{n_i}$. When doubled in radius, this ball may ``spill over" to the $j$th end, with measure bounded below by $c_{n_j} R^{n_j}$. These are not comparable for $R \to \infty$, so $\M$ will fail to be doubling in this situation. This geometric property means that many standard strategies in harmonic analysis for proving $L^p$ boundedness need to be avoided or adapted. 
  
 The metric measured spaces which satisfy the doubling condition are called 
 homogeneous spaces. It is rather confusing nomenclature as the doubling condition does not imply uniformity of the volume behaviour over the whole space. Nevertheless it is a very common nomenclature and we shall use it as well. 
 The notion of a homogeneous space was introduced by R. Coifman and G. Weiss in  \cite{CoW} almost a half a century ago. Since then the doubling condition has been a central point of modern harmonic analysis and heat kernel theory. The notion is especially significant in the theory of singular integral operators \cite{CZ}.

At the same time the doubling condition is commonly considered as a technical assumption 
which is not necessarily very natural and it is often not clear it is essential. In this context let us mention the results obtained by Stein in 
\cite{St} which assert that for any $1\le p \le \infty$  there exists a uniform bound for $L^p(\R^n) $ norm of the Riesz transform  which is independent of $n$.
Clearly the optimal constant $C$ in the doubling condition \eqref{doub}
is equal to $2^n$ and this suggests that this condition should not play an essential role in the proof of the continuity of the Riesz transform. 
Further results concerning dimension-free bounds for the Riesz transform can be 
found in \cite{CMZ} and in references therein. 

Another result which sheds some light on relation between the Riesz transform 
and condition \eqref{doub} was obtained by Hebisch and Steger. In \cite{HeSt}
they proved the boundedness of the Riesz transform on all $L^p$ spaces $1<p<\infty$ for a class of Laplace operators acting on some Lie groups of exponential growth where of course the doubling condition fails.  Another example of interest attracted by nonhomogeneous spaces comes from celebrated result of Nazarov, Treil and  Volberg \cite{NTV}. They studied  Calder\'on-Zygmund operators in nonhomogeneous  setting  and conclude that ``The doubling condition is superfluous for most of the classical theory''. Well-known results going in the similar direction were obtained in various papers of Tolsa, see for example \cite{To1, To2}. 

Our result provides another interesting example of  an operator of Calder\'on-Zygmund type in a non-homogeneous  setting.
 However, the nature of the spaces we consider, and our methods of proof, are completely different than those considered in 
 \cite{St, CMZ, HeSt, NTV, To1, To2}. The most significant difficulty in our investigation is to understand the kernel of the Riesz transform far from the diagonal, which was not an essential difficulty in the papers mentioned above.


\bigskip

Before we state our main result let us recall the notion of connected sum of smooth manifolds in a more precise manner.
We refer the readers
to \cite{GS, GS2} for further discussion of this definition. 

\begin{definition}\label{def:connected-sum} Let $\mathcal{V}_i$, for $i=1,\cdots, l$ be a family of complete connected non-compact Riemannian manifolds of the same dimension. We say that 
a Riemannian manifold 
$ \mathcal{V}$ is a connected sum of $\mathcal{V}_1,\ldots,\mathcal{V}_l$
and write 
$$
\mathcal{V}=\mathcal{V}_1\# \mathcal{V}_2 \#\ldots \#\mathcal{V}_l
$$
if for some compact subset $K \subset \mathcal{V}$ the exterior 
$ \mathcal{V}\setminus K$ is a disjoint union of connected open sets $\mathcal{V}_i$, $i=1,\cdots, l$, such that each $\mathcal{V}_i$ is isometric to $\mathcal{V}_i \setminus K_i$
for some compact sets $K_i \subset \mathcal{V}_i$. We call the subsets $\mathcal{V}_i$ the \emph{ends} of $\mathcal{V}$. 
\end{definition}

Our approach provides a flexible tool to study analysis on connected sum of smooth manifolds and this is another motivation for our study. 
Consider again the family $\R^{n_i} \times \mathcal{M}_i$
for $i=1,\cdots, l$ where $\mathcal{M}_i$ are compact manifolds such that 
$\dim \mathcal{M}_i + n_i = N$ and $n_i\ge 3$  for each $i$. In the terms of the above definition we can consider manifolds with $l$ ends of the form
\begin{equation}\label{defM}
\mathcal{M}=(\R^{n_1} \times \mathcal{M}_1) \# (\R^{n_2} \times \mathcal{M}_2) 
\# \ldots \# (\R^{n_l}\times \mathcal{M}_l).
\end{equation} 
The main result of this paper can be stated in the following way 
 
 \begin{theorem}\label{thm:main}
 Suppose that  $\mathcal{M}=(\R^{n_1} \times \mathcal{M}_1) \# 
\ldots \# (\R^{n_l}\times \mathcal{M}_l)$ is a manifold with $l \geq 2$ ends  defined by \eqref{defM}, with $n_i \geq 3$ for each $i$. 
Then  the Riesz transform   $\nabla \Delta^{-1/2}$ defined on $\M$ is bounded on $L^p(\M)$ if and only if 
$1 < p < \min\{n_1,\cdots, {n_l}\}$.
That is, there exists $C$ such that 
 $$\big\| \, |\nabla \Delta^{-1/2} f| \, \big\|_p\le C \|f\|_p,\ \forall f\in L^p(X,\mu)$$
 if and only if $1 < p < \min\{n_1,\cdots, {n_l}\}$.
 In addition the Riesz transform   $\nabla \Delta^{-1/2}$ is of weak type 
 $(1,1)$.
\end{theorem}

Let us compare this theorem to \cite[Proposition 3.3]{Ca2}:

\begin{proposition}\label{prop:Carron}(Carron) 
Let $(V_1, g_1), \dots, (V_k, g_k)$ be complete Riemannian manifolds with non-negative Ricci curvature. Assume that for all $i$
we have the volume lower growth bound:
$$
\operatorname{vol} B(o_i, R) \geq C R^\nu.
$$
Assume also that $\nu > 3$ and that $\nu/(\nu-1) < p < \nu$. Then on any manifold isometric at infinity to a disjoint union of the $(V_i, g_i)$, the Riesz transform is bounded on $L^p$. 
\end{proposition}
Proposition~\ref{prop:Carron} covers the case that all the $\M_i$ in Theorem~\ref{thm:main} have nonnegative Ricci curvature (for example, tori or spheres). Actually, a closer reading of Section 3.3 of \cite{Ca2} shows that the proof extends automatically to the case of arbitrary compact $\M_i$. We also note that unboundedness for $p \geq \nu$ is shown in \cite[Theorem C]{Ca1}. 
So the improvement in Theorem~\ref{thm:main} consists in extending the range of $p$ down to $p = 1$ (including the weak-type result at $p=1$) as well as allowing the case $n_i = 3$. We note that the methods of proof are different, though both use heat kernel estimates. Carron's proof proceeds via analysis of the Poisson kernel, while here we return to the method of \cite{CCH} and study the low-energy asymptotics of the resolvent. 
Note also that in the setting of non-doubling manifolds the $L^p$-boundedness and weak type $(1,1)$ result for the Riesz transform does not follow automatically from any known results. In fact for Schr\"odinger operators with non-zero potentials there are examples when the Riesz transform is not bounded for $p$ close to $1$; see e.g. \cite{GH1, GH2, HS}. For Laplace-Beltrami operators which we consider here the Riesz transform is 
always bounded for all $1<p \le 2$. To our knowledge our result is the only example where boundedness has been proved  for all $1<p \le 2$ in the non-doubling setting.  

In Section \ref{gen} below we extend the results to a larger class of manifolds with ends
and we replace $\R^{n_i}\times \mathcal{M}_i$ by a class of manifolds which 
includes Lie groups of polynomial growth and  operators with periodic coefficients. 

Our proof is based on techniques from \cite{CCH}. Like \cite{CCH}, the proof is a kind of synthesis of harmonic and microlocal analysis, but here we mostly use harmonic analysis techniques (heat kernel estimates, spectral multipliers) and avoid the Melrose-style language of compactifications, blowups, etc. Our approach is quite flexible and has significant potential for further application. We hope the paper can attract readers from both backgrounds. 


\section{Preliminaries}

\subsection{Resolvent of the Laplacian}

Similarly as in \cite{CCH} our approach is primarily based on the resolvent of the Laplace operator $\Delta$.  Recall that $\Delta$ is a positive defined self-adjoint operator and by spectral theory 
the Riesz transform  $\nabla \Delta^{-1/2}$ can be represented as 
\begin{equation}
\nabla \Delta^{-1/2}=\frac{2}{\pi} \nabla \int_0^{\infty} (\Delta + k^2)^{-1} \, dk.
\label{Whole-Riesz}\end{equation}
Next we split the operator $ \nabla \Delta^{-1/2}$ into two parts corresponding to low and high energies. That is, for some small exponent 
$k_0$ to be determined later, we define 
\begin{equation}
F_<(\lambda) = \frac{2}{\pi} \int_0^{k_0} (\lambda^2+k^2)^{-1} dk, \quad F_>(\lambda) = \frac{2}{\pi} \int_{k_0}^\infty (\lambda^2+k^2)^{-1} dk.
\label{F<>}\end{equation}
so that $\Delta^{-1/2}=F_<(\sqrt{\Delta})+F_>(\sqrt{\Delta})$. Hence the Riesz transform can be represented as 
$$
\nabla \Delta^{-1/2}=\nabla F_<(\sqrt{\Delta}) +\nabla F_>(\sqrt{\Delta}). 
$$
We shall show that the Riesz transform localized to low energies 
$\nabla F_<(\sqrt{\Delta})$ is bounded in the range of $L^p$ spaces described in 
Theorem \ref{thm:main}, whereas the high energy part $\nabla F_>(\sqrt{\Delta})$ is bounded for all $1<p< \infty$. The most essential part of our discussion is to construct and understand behaviour of the resolvent 
$ (\Delta + k^2)^{-1}$ for $0 < k \le k_0$. 

For later use note that by performing the integration in \eqref{F<>} we find that 
$$
F_>(\lambda) =  \frac{2} {\pi} \lambda^{-1}\tan^{-1}\left( \frac{\lambda}{k_0} \right) .
$$
From this we see that
\begin{equation}
F_>(\lambda) \sim \begin{cases} \frac1{k_0}, \quad \lambda \to 0 \\ \ \lambda^{-1}, \quad \lambda \to \infty \end{cases}
\end{equation}
and moreover
\begin{equation}
F_>(\lambda) \text{ is a symbol of order } -1 \text{ as a function of } \lambda.
\label{symbol}\end{equation}

\subsection{The resolvent on a product space $\R^{n_i} \times \mathcal{M}_i$}
Let  $\Delta_{\R^{n_i} \times \mathcal{M}_i}$ be the Laplacian on $\R^{n_i}\times {\mathcal{M}_i}$. 
In what follows we will need a several straightforward estimates for the  heat kernel and resolvent corresponding to this operator.

We use $x_i$ to denote a Euclidean coordinate in $\R^{n_i}$ and write $z_i = (x_i, y_i)$ for a coordinate on the $i$th end, where $y_i \in \mathcal{M}_i$. We sometimes drop the subscript from $x_i$ and $n_i$ where no confusion seems possible. We also use primed/unprimed coordinates to refer to the left/right coordinate on the double space $\mathcal{M}^2$. By $d(z,z')$ we denote the Riemannina distance between points $z$ and $z'$. 

The resolvent on the product space $\R^{n_i} \times \mathcal{M}_i$ play an essential  role in our approach. It also will be an ingredient of the parametrix of the following section, and to analyze this parametrix effectively, we need several straightforward estimates on this resolvent. 
We start with estimates for the heat kernel. These are particularly straightforward on a Riemannian product such as $\R^{n_i} \times \mathcal{M}_i$, because the heat kernel on the product space is just the pointwise product of the heat kernels on the two factors. Moreover, the heat kernel on $\R^{n_i}$ is completely explicit, the well-known Gaussian
\begin{equation}
e^{-t\Delta_{\R^{n}}}(x,x') = \frac1{(4\pi t)^{n/2}}\exp \left( \frac{|x-x'|^2}{4t} \right),
\label{Rn-heatkernel}\end{equation}
while the heat kernel on $\mathcal{M}_i$ obeys Gaussian estimates for small $t$, and for large~$t$, is constant up to an exponentially decaying error: 
\begin{equation}
\big| e^{-t\Delta_{\mathcal{M}_i}}(y,y') \big| \leq \begin{cases} 
C t^{-N_i/2} \exp \Big( \frac{d(y,y')^2}{ct} \Big), \quad t \leq 1, \\
\frac1{\vol(\mathcal{M}_i)} + O(e^{-\mu_1 t}), \quad t \geq 1, \end{cases}
\label{Mi-heatkernel} \end{equation}
where $\mu_1$ is the first positive eigenvalue of the Laplacian on $\mathcal{M}_i$ and $N_i=\mbox{dim}\,\M_i=N-n_i$. Here the constant $(\vol (\mathcal{M}_i))^{-1}$ should be understood as the kernel of the orthogonal projection onto the constant functions (i.e.\ the zero eigenspace of the Laplacian) on $\mathcal{M}_i$. Moreover, each satisfies spatial derivative estimates of the form  
\begin{equation}
\big| \nabla e^{-t\Delta_{\R^{n}}}(x,x') \big|  \leq C_c \frac1{(4\pi t)^{(n+1)/2}}\exp \left( \frac{|x-x'|^2}{ct} \right), \quad c \in (0, 4),
\label{Rn-heatkernel2}\end{equation}
and  
\begin{equation}
\big| \nabla e^{-t\Delta_{\mathcal{M}_i}}(y,y') \big| = O(e^{-\mu_1 t}), \quad t \geq 1.
\label{Mi-heatkernel2} \end{equation}
It follows that the heat kernel on the product satisfies 
 so called Gaussian bounds 
\begin{eqnarray}\label{gaus}
e^{-t\Delta_{\R^{n_i} \times \mathcal{M}_i}} (z,z') \le C(t^{-n_i/2}+t^{-N/2})\exp\left(-c\frac{d(z,z')^2}{t}\right)
\end{eqnarray}
and 
\begin{eqnarray}\label{g}
	\big|\nabla e^{-t\Delta_{\R^{n_i} \times \mathcal{M}_i}} (z,z') \big| \le C\left(t^{-(n_i+1)/2}+t^{-(N+1))/2}\right)\exp\left(-c\frac{d(z,z')^2}{t}\right).
\end{eqnarray}
In addition the following lower-bounds are also valid 
\begin{eqnarray}\label{low}
e^{-t\Delta_{\R^{n_i} \times \mathcal{M}_i}} (z,z') \ge C'(t^{n_i/2}+t^{N/2})\exp\left(-c'\frac{d(z,z')^2}{t}\right).
\end{eqnarray}
In fact it  was shown in \cite{CS1} that estimate \eqref{g} implies both \eqref{gaus}
and \eqref{low} in the setting of manifolds which satisfies the doubling condition. 

\begin{remark}
To verify estimate \eqref{g} it is enough to prove 
that $$|\nabla e^{-t\Delta}(z,z')| \le C\left(t^{(n_i+1)/2}+t^{(N+1))/2}\right).$$ Then the Gaussian term can be added automatically, see \cite{ACDH} and \cite{Si}. 
It is known that \eqref{g} holds for smooth compact manifolds, Lie groups with polynomial growth and divergence form operator with periodic coefficient 
acting on $\R^n$, see \cite{Sal1, ERS}. Note also that \eqref{g} holds on any Cartesian product of these spaces. 
\end{remark}

We now use these heat kernel estimates to derive resolvent estimates, using the identity
\begin{equation}
(\Delta_{\R^{n_i} \times \mathcal{M}_i} + k^2)^{-1} = \int_0^\infty e^{-tk^2} e^{-t \Delta_{\R^{n_i} \times \mathcal{M}_i}} \, dt.
\label{heat-to-res}\end{equation}

It is convenient to introduce the notation 
$$
 L_a(r)=r^{1-a/2}K_{|a/2- 1|}(r), \quad a \geq 1,
$$ 
where $K_{|a/2- 1|}$ is the  modified Bessel function which decays to zero 
at when $r \to \infty$ --- 
see \cite[\S 9.6.1 p. 374]{AS} or \cite[\S 1.14 p. 16]{Tr}.
It is easy to check, see for example \cite{HS},  that the function $L_a$ satisfies the following 
differential equation
\begin{equation}
f''+\frac{a-1}{r}f'= f.
\label{kl}\end{equation}
Note that for every $a\ge 1$ and for some positive constant $C_a$ 
\begin{equation}
  \int_0^\infty t^{-a/2}e^{-tk^2} \exp\left(-\frac{r^2}{4t}\right)  \, dt
	= C_a k^{a-2} L_a(kr)
	\label{Bess-int}
	\end{equation}
Indeed one can verify the above equality using the fact that $L_a$ spans the linear space of all functions which satisfy \eqref{kl} and  decays to zero 
at $+\infty$. Let us recall the following standard asymptotic   valid for all exponents $a>2$. 
\begin{equation*} 
L_a(r)\approx \left\{ \begin{array}{ll}
r^{2-a}  & \mbox{if} \quad r \le 1\\
r^{(1-a)/2} e^{-r}    &    \mbox{if} \quad r > 1,
\end{array}
\right.
\end{equation*}
see for example \cite{HS}. 

Note that if  $a \geq 3$ then  $C_{a,c}r^{2-a} e^{-r}  \le  L_a(r) \le C'_ar^{2-a} e^{-cr}$, $0<c<1$ so  it follows from the above asymptotic and \eqref{Bess-int} that  
\begin{equation}\label{re-sup}
 (\Delta_{\R^{n_i} \times \mathcal{M}_i} + k^2)^{-1}(z,z')  \leq C(d(z,z')^{2-N} + d(z,z')^{2-n_i} )
\exp(-ckd(z,z'))
\end{equation}
and 
\begin{equation}\label{re-inf}
 (\Delta_{\R^{n_i} \times \mathcal{M}_i} + k^2)^{-1}(z,z')  \ge C'(d(z,z')^{2-N} + d(z,z')^{2-n_i} )
\exp(-kd(z,z')).
\end{equation}
In addition by \eqref{g} 
\begin{equation}\label{g-res}
\Big| \nabla (\Delta_{\R^{n_i} \times \mathcal{M}_i} + k^2)^{-1}(z,z') \Big| \leq C(d(z,z')^{1-N} + d(z,z')^{1-n_i} ) 
\exp(-ckd(z,z'))
\end{equation}

Before we state our first lemma it is convenient to discuss some more general notions of heat kernel theory.  Consider now a manifold $(\mathcal{V},\mu)$, where $\mu$ is a smooth non-vanishing  measure $\mu$ and the corresponding Laplace-Beltrami operator $\Delta_\mathcal{V}$,  determined by the relation 
\begin{equation}
\int_\mathcal{V} g\Delta_\mathcal{V}f d\mu =\int_\mathcal{V} \langle \nabla f, \nabla g \rangle \d\mu.
\label{LapV}\end{equation}
We say that the heat kernel corresponding  to the operator $\Delta_\mathcal{V}$
satisfies Gaussian bounds if
\begin{eqnarray}\label{gaus1}
e^{-t\Delta_\mathcal{V}} (z,z') \le C\mu(B(x,r)^{-1/2}\exp\left(-c\frac{d(z,z')^2}{t}\right).
\end{eqnarray}
We already have pointed out that such estimates holds for Cartesian product of $\R^{n_i}\times \mathcal{M}_i$ in \eqref{gaus}.

Now we are able to  state a standard spectral type multipliers result   
which we need in  what follows and which is related to the explicit formula for $F_{>}(\lambda)$ above. 

\begin{lemma}\label{l2.1}
Let $\Delta_\mathcal{V}$ be the Laplace-Beltrami operator acting on a complete Riemannian manifold $\mathcal{V}$ with a smooth measure $\mu$.
Let $F$ be the function defined by $F(\lambda) = \pi/2 - \tan^{-1}(\lambda)$. 
If the space $(\mathcal{V}, \mu)$ satisfies the doubling condition and if the heat kernel $e^{-t\Delta_{\mathcal{V}}}$ satisfies \eqref{gaus1},  then the operator $F(\frac{\sqrt{\Delta_V}}{a})$ defined initially on $L^2(\Delta_\mathcal{V}, \mu)$
via spectral theorem can be extended to bounded operator on all $L^p(\mathcal{V}, \mu)$ spaces and 
$$
\Big\|F\Big(\frac{\sqrt{\Delta_\mathcal{V}}}{a}\Big)\Big\|_{p\to p} \le C_a < \infty 
$$
for all $a>0$ and $1\le p \le \infty$. 
\end{lemma}

\begin{proof}
It is not difficult to note that the doubling condition \eqref{doub}
implies that there exists a constant $n$ such that 
\begin{eqnarray*}
\mu(B(x,rt))\leq C t^n \mu(B(x,r))\quad \forall\,r>0 , t>1.  
\end{eqnarray*}	
Now a standard spectral multiplier result yields that if the space $(\mathcal{V},d,\mu)$
satisfies the above conditions  and the Gaussian bounds \eqref{gaus1} hold then for any Borel function $G\in C^{[n/2]+1}$ supported in the interval $[-1,1]$ the operator $G({\sqrt{\Delta_\mathcal{V}}})$ defined initially on $L^2$ via spectral theory satisfies 
$$
\Big\| G\Big({\sqrt{\Delta_\mathcal{V}}}{}\Big) \Big\|_{1\to 1} \le C \|G\|_{C^{[n/2]+1}}
$$
 see e.g. 
\cite{DOS, CSi}. Using dyadic decomposition of $G$ 
it follows that for any $\epsilon >0$ and any function $G\colon [0,\infty) \to \R$
\begin{equation}
\Big\| G\Big({\sqrt{\Delta_\mathcal{V}}}{}\Big) \Big\|_{1\to 1} \le C \max_{0\le k  \le [n/2]+1}\sup_\lambda |G^{(k)}(\lambda)(1+\lambda)^{k+\epsilon}|.
\label{sm}\end{equation}
Now if we set  $G(\lambda) =\pi/2- \tan^{-1}(\frac{\lambda}{a})$ then    
$$
|G^{(k)}(\lambda) (1+\lambda)^{-(k+1)}| \le C_{k,a} \quad \mbox{for} \quad k=0,1,\ldots
 $$ 
so the RHS of \eqref{sm} is finite with $\epsilon = 1$. It follows that $G({\sqrt{\Delta_\mathcal{V}}})=F(\frac{\sqrt{\Delta_\mathcal{V}}}{a})$ satisfies the statement in the lemma. 
\end{proof}

\begin{remark}
The part of the argument from \cite{DOS, CSi} to verify Lemma \ref{l2.1} is in fact quite simple. Most of technical difficulties in \cite{DOS, CSi} arise because the obtained spectral multipliers are sharp. This is not essential to prove Lemma \ref{l2.1}. 

One can also use a simplified version of the proof of 
Theorem 2.4 in \cite{GHS} to prove the lemma. In fact, as noted in \cite[Section 2]{GHS}, the spectral projection estimate (2-5) from \cite[Section 2]{GHS} holds for arbitrary $\lambda_0$ for any manifold with bounded geometry and positive injectivity radius. The proof can be modified to only use (2-5) rather than the stronger estimate (2-3) at the cost of one additional derivative assumed on $F$ (just integrate by parts in (2-12) to replace the spectral measure by the spectral projection). This manipulation is harmless here as our $F$ is infinitely differentiable. 
\end{remark}


\subsection{Sobolev inequality}
We will use several times the following result of Cheeger-Gromov-Taylor:

\begin{proposition}[\cite{CGT}, Prop. 1.3]\label{prop:CGT}
Let $\M$ be a manifold of dimension $N$ with $C^\infty$ bounded geometry (all derivatives of the curvature tensor are uniformly bounded, and 
the injectivity radius is positive). Then there is an $r$ depending only on the bound on curvature and the injectivity radius
such that for any $x_0 \in \M$, and $k = [N/4] + 1$, 
\begin{equation}
|g(x_0)| 
\leq C \Big(  \| g \|_{L^2(B_r(x_0))} +  \| \Delta g \|_{L^2(B_r(x_0))} + \dots +  \| \Delta^k g \|_{L^2(B_r(x_0))} \Big).
\label{CGT-est}\end{equation}
As an immediate consequence, the operator $(1 + \Delta)^{-k}$ is bounded from $L^2(\M)$ to $L^\infty(\M)$. 
\end{proposition}

\begin{proof} We briefly sketch the proof. The parameter $r$ is chosen so that the metric in the ball of radius $r$ around each point in normal coordinates is uniformly bounded in $C^\infty$. 
Let $Q$ be a parametrix for $1 + \Delta$, supported in $\{ d(x,y) \leq r/k \}$ and with seminorms uniform over $\M$. This is possible due to the uniform boundedness of the metric in normal coordinates. Then we have
$$
I =  Q^k(I+\Delta)^{k}  + R,
$$
where $R$ is a operator with smooth kernel supported in $\{ d(x,y) \leq r \}$ and uniformly bounded in $C^\infty$. In particular the kernel of $R$ is in $L^\infty(\M \times \M)$. So  we have 
\begin{equation}
g(x) = \int Q^k(x, y) \big((I+\Delta)^{k} g \big)(y)  dy + \int R(x,y) g(y) dy.
\label{omega-psi}\end{equation}
Since $Q^k(Q^k)^*$ is a pseudodifferential operator of order $-4k < -N$, with seminorms uniform over $\M$, it has a continuous $L^\infty$ kernel,  so it maps $L^1 \to L^\infty$. Therefore $Q^k$ itself maps $L^2 \to L^\infty$. Hence, using the support condition on the kernel, 
$Q^k(x_0, \cdot)$ has a uniformly bounded $L^2$ norm, so $|(Q^k  \big((I+\Delta)^{k} g \big))(x_0)|$ is uniformly bounded by $C \| \big((I+\Delta)^{k} g \big) \|_{L^2(B(x_0, r))}$. For a similar reason, we have  $| (R g)(x_0) | \leq C \| g \|_{L^2(B(x_0, r))}$.  This completes the proof. 
\end{proof}

\begin{remark}\label{rem:forms} Exactly the same argument applied to $q$-forms shows that, if $\Delta_q$ is the Laplacian on $q$-forms on $\M$, then $(1 + \Delta_q)^{-k}$ maps $L^2$  to $L^\infty$. 
We use this for $q=1$ in Section~\ref{sec:high}. 
\end{remark}
\begin{remark}\label{rem:derivs}
Another easy corollary is that, if $W_1, \dots, W_s$ are $C^\infty$ vector fields on $\M$, uniformly bounded in every $C^m$ norm, then we have an estimate 
\begin{equation*}
\| W_1 \dots W_s g \|_{L^\infty(\M)} \leq C \Big(  \| g \|_{L^2} +  \| \Delta g \|_{L^2} + \dots +  \| \Delta^{k+k'} g \|_{L^2} \Big),
\end{equation*}
provided that $2k' \geq s$. To see this we simply apply the differential operator $W_1 \dots W_s$ to both sides of \eqref{omega-psi} and argue as before. 
\end{remark}

\subsection{A key lemma}
For each end $\R^{n_i} \times \mathcal{M}_i$ we choose a point $z_i^\circ$ such that $z_i^\circ \in K_i$ where $K_i$ are  the sets used the in the definition of connected sum above. 
 
The following result is crucial for our parametrix construction (compare Lemma 3 \cite{CCH}). 

\begin{lemma}\label{uv} Assume that each $n_i$ is at least $3$. 
Let $v \in C^\infty_c(\M; \R)$. Then there is a function $ u\colon \mathcal{M}\times \R_+ \to \R$ such that 
$(\Delta +k^2)u = v$ and such that, 
on the $i$th end we have: 
\begin{equation}\begin{gathered} 
|u(z, k)| \leq C \ang{d(z_i^\circ,z)}^{-(n_i-2)} \exp(-{kd(z_i^\circ,z)}) \quad \forall z\in \R^{n_i}\times \mathcal{M}_i \\
|\nabla u(z,k)| \leq C  \ang{d(z_i^\circ,z)}^{-(n_i-1)} \exp(-{kd(z_i^\circ,z)})  \quad \forall z\in \R^{n_i}\times \mathcal{M}_i
\end{gathered}\label{u-estimates}\end{equation}
for some $c, C > 0$. For any $k_0 > 0$ we also have pointwise estimates
\begin{equation} 
\| u(\cdot,k) - u(\cdot, 0) \|_{L^\infty(\M)} \leq C k  , \quad k \leq k_0
\label{u-kminus0}\end{equation}
and 
\begin{equation}
\Big \| \nabla \big( u(\cdot,k) - u(\cdot, 0) \big) \Big\|_{L^\infty(\M)} \leq C k, \quad k \leq k_0.
\label{u-kminus0-grad}\end{equation}

\end{lemma}

\begin{proof}

We use Corollary 4.9 of \cite{GS} to see that the heat kernel applied to $v$ is in $L^\infty$ for short time and decays as $O(t^{-n_{min}/2})$ pointwise for times $t \geq 1$, uniformly over the manifold. From the assumption that each $n_i$ is at least 3, this is integrable in time. 
Therefore the solution to the heat kernel with initial condition $v$, namely $e^{-t\Delta} v$, is bounded in $L^\infty$ by $\| v \|_\infty$ for times $t \leq 1$ and by $C \| v \|_1 t^{-3/2}$ for $t \geq 1$. 
It follows that the integral 
$$
\int_0^\infty e^{-tk^2} e^{-t\Delta} v \, dt 
$$
converges in $L^\infty$, uniformly in $k \geq 0$, to a solution $u$ of the equation $(\Delta + k^2) u = v$. 
Moreover, we have 
$$
\Delta u = v - k^2 u,
$$
implying that $\Delta u \in L^\infty$, uniformly in $k \in [0,1]$. We can repeatedly apply factors of $\Delta$ to $u$ and use the equation to write this in terms of $u$ and $v$, showing that $\Delta^m u \in L^\infty$ uniformly in $k \in [0,1]$ for each positive integer $m$. Applying Remark~\ref{rem:derivs} we see that $u \in C^\infty$ uniformly in $k \in [0,1]$ on every compact subset of $\M$.

Next we consider $u$ on each end.  Let $\zeta_i \in C^\infty(\M)$ be a function such that  supp~$\zeta_i$ is contained entirely in $\R^{n_i}\times {\mathcal{M}_i}\setminus K_i$ and $1 -\zeta_i$ considered as a function on $\R^{n_i}\times {\mathcal{M}_i}$ is compactly supported. We assume also that $\zeta_i = 0$ on the support of $v$. 
(We extend $\zeta_i$ to function on $\R^{n_i}\times {\mathcal{M}_i}$ by defining it to be equal zero on $K_i$ which is the part of $\R^{n_i}\times {\mathcal{M}_i}$ which was removed before connecting it with the rest of the manifolds $\M$.) 
Write $\Delta_{\R^{n_i} \times \mathcal{M}_i}$ for the Laplacian on $\R^{n_i}\times {\mathcal{M}_i}$. Now consider the functions $u \zeta_i$, viewed as a function on $\R^{n_i}\times {\mathcal{M}_i}$, and 
$$\tilde u_i(z,k) := ({\Delta_{\R^{n_i} \times \mathcal{M}_i}} +k^2)^{-1}\Big( (\Delta +k^2)(\zeta_i u(z,k)) \Big).$$
Notice that the action of $\Delta +k^2$ and $\Delta_{\R^{n_i} \times \mathcal{M}_i}+k^2$ on $\zeta_i u$ is the same. So applying ${\Delta_{\R^{n_i} \times \mathcal{M}_i}} + k^2$ to $\tilde u_i$ gives the same result as applying ${\Delta_{\R^{n_i} \times \mathcal{M}_i}} + k^2$ to $u \zeta_i$, namely $(\Delta + k^2) (u \zeta_i)$. Since both $\tilde u_i$ and $(\Delta + k^2) (u \zeta_i)$ are in $L^2$, and ${\Delta_{\R^{n_i} \times \mathcal{M}_i}} + k^2$ is injective on $L^2(\R^{n_i}\times {\mathcal{M}_i})$, we conclude that $\tilde u_i = u \zeta_i$. It follows that, on the $i$th end, $u$ is given by the resolvent $({\Delta_{\R^{n_i} \times \mathcal{M}_i}} + k^2)^{-1}$ applied to a smooth compactly supported function $f(\cdot, k)$ on $\R^{n_i} \times \mathcal{M}_i \times \R_+$. 
Now estimates \eqref{u-estimates} follows from \eqref{re-sup} and \eqref{g-res}.

To prove \eqref{u-kminus0}, we again use the result from Grigoryan-Saloff-Coste that the $L^1 \to L^\infty$ norm of the heat kernel on $\M$ is bounded by $C t^{-3/2}$ for $t \geq 1$. On the other hand, the $L^\infty \to L^\infty$ norm of the heat kernel on $\R^{n_i} \times \mathcal{M}_i$ is bounded by $1$ (maximum principle) for all times.  So the $L^\infty$ norm of $u(\cdot, k) - u(\cdot, 0)$ may be bounded by 
$$
\int_0^1 (1 - e^{-tk^2}) dt \times \| v \|_{L^\infty} + C \| v \|_{L^1(\M)} \int_1^\infty (1 - e^{-tk^2}) t^{-3/2} \, dt .
$$
The first integral is clearly $O(k^2)$. By a change of variable to $t' = tk^2$ we can write the integral in the second term as 
$$
k \int_{k^2}^\infty (1 - e^{-t'}) {t'}^{-3/2} \, dt' \leq k \int_{0}^\infty (1 - e^{-t'}) {t'}^{-3/2} \, dt' 
$$
which is also clearly $O(k)$. 

We now prove \eqref{u-kminus0-grad}. We use the identities $(\Delta + k^2) u(k) = \Delta u(0) = v$ to obtain $\Delta (u(k) - u(0)) = -k^2 u(k)$. As we showed above that $u(k)$ is in $C^\infty(\M)$ uniformly in $k$, this shows that  $\Delta(u(k) - u(0))$ is $O(k^2)$ in $L^\infty(\M)$; indeed, $\Delta^j (u(k) - u(0))$ is $O(k^2)$ in $L^\infty(\M)$ for any positive integer $j$. Now this in combination with \eqref{u-kminus0} and Remark~\ref{rem:derivs} (with $s=1$) shows that $\nabla (u(k) - u(0))$ is  $O(k)$ in $L^\infty(M)$.
\end{proof}

\section{Low Energy Parametrix}
Following \cite{CCH}, we write down a parametrix $G = G(k)$  for the resolvent $(\Delta + k^2)^{-1}$ on $\mathcal{M}$, in the low energy case $k \leq 1$.  To do this, 
let $\phi_i \in C^\infty(\M)$ be a function such that  
\begin{itemize}
\item supp~$\phi_i$ is contained entirely in $\R^{n_i}\times {\mathcal{M}_i}\setminus K_i$ (viewed as a subset of $\M$), and 
\item $\phi_i$, viewed as a function on $\R^{n_i}\times {\mathcal{M}_i}$, equals $1$ outside a compact set. 
\end{itemize}
Next, let $v_i = -\Delta \phi_i$, which is compactly supported, and let $u_i$ be the function on $\M \times \R_+$ given by Lemma~\ref{uv} applied to $v_i$. Notice that $\Phi_i := u_i(\cdot, 0) + \phi_i$ is harmonic. 

Let $G_{int}(k)$ be an interior parametrix for the resolvent, supported close to a compact subset $K_\Delta$ of the diagonal of $\mathcal{M}^2$, say
$$
K_\Delta = \{ (z, z) \mid z \in K \},
$$
where $K$ is as in Definition~\ref{def:connected-sum}, 
 and agreeing with the resolvent of $\Delta_{\R^{n_i} \times \mathcal{M}_i}$  in a (smaller) neighbourhood of $K_\Delta$, intersected with the support of $\nabla \phi_i(z) \phi_i(z')$.
Then let $z_i^\circ\in {\R^{n_i} \times \mathcal{M}_i}$ be a point outside the support of function $\phi_i$.

The parametrix $G(k)$ will be defined through its Schwartz kernel, which is a locally integrable function on $\mathcal{M}^2$. 
Notice that $\mathcal{M}^2$ has $l^2$ ends, namely $\R^{n_i} \times \mathcal{M}_i \times \R^{n_j} \times \mathcal{M}_j$, where $i, j \in \{ 1 \dots l \}$. We define $\tilde G(k) = G_1(k) + G_2(k) + G_3(k)$, where 
\begin{equation}\begin{aligned}
G_1(k) &= \sum_{i=1}^l (\Delta_{\R^{n_i} \times \mathcal{M}_i} + k^2)^{-1}(z, z') \phi_i(z) \phi_i(z') \\
G_2(k) &= G_{int}(k) \Big( 1 - \sum_{i=1}^l \phi_i(z) \phi_i(z') \Big) \\
G_3(k) &= \sum_{i=1}^l  (\Delta_{\R^{n_i} \times \mathcal{M}_i} + k^2)^{-1}(z_i^\circ, z') u_i(z,k) \phi_i(z').
\end{aligned}\label{G(k)}\end{equation}
Thus $G_1(k)$ is supported on the `diagonal ends' $(i,i)$, while $G_2(k)$ is supported on a compact subset of $\mathcal{M}^2$. However,  the support of $G_3(k)$ extends to all $l^2$ ends, as $u_i$ is defined globally on $\mathcal{M}$. The parametrix $G(k)$ will be defined by $G(k) = \tilde G(k) + G_4(k)$ where $G_4(k)$ is defined below. 

We define the error term $\tilde E(k)$ by 
$$
(\Delta + k^2) \tilde G(k) = \Id + \tilde E(k).
$$
It is important to compute the order of vanishing of the error kernel $\tilde E(k)$ as the right variable tends to infinity. 
Using $$(\Delta + k^2) u_i = v_i = -\Delta \phi_i$$
we compute explicitly 
\begin{equation}\begin{gathered}
\tilde E(k)(z,z') = \sum_{i=1}^l \Bigg(  \nabla \phi_i(z) \phi_i(z') \Big( \nabla_z (\Delta_{\R^{n_i} \times \mathcal{M}_i} + k^2)^{-1}(z, z')  - \nabla_z G_{int}(z,z') \Big) \\
+ \phi_i(z') v_i(z) \Big(  -(\Delta_{\R^{n_i} \times \mathcal{M}_i} + k^2)^{-1}(z, z') + G_{int}(z,z') +  (\Delta_{\R^{n_i} \times \mathcal{M}_i} + k^2)^{-1}(z_i^\circ, z') \Big) \Bigg) \\
+ \Big( (\Delta + k^2) G_{int}(z,z') - \delta_{z'}(z) \Big) \Big( 1 - \sum_{i=1}^l \phi_i(z) \phi_i(z') \Big).
\end{gathered}\label{tilde-E}\end{equation}
Consider the smoothness and decay properties of this error term. 

\begin{itemize}
\item
The first line of the RHS of \eqref{tilde-E} is smooth, since by construction, the interior parametrix $G_{int}$ agrees with the resolvent near the diagonal, on the $i$th end and on the support of $\nabla \phi_i(z) \phi_i(z')$. The first line on the RHS is also compactly supported in the left variable $z$, and, by \eqref{re-sup}, is $O(d(z_i^\circ,z')^{(1-n)} \exp(-{kd(z_i^\circ,z')})$ for $k \leq 1$. 
\item
The second line is smooth across the diagonal for the same reason, and also vanishes to order $n-1$ as $d(z_i^\circ,z') \to \infty$, using \eqref{g-res}.
\item
The third line is smooth and compactly supported. 
\end{itemize}

It follows that the kernel $\tilde E(k)$ is smooth. Moreover, letting 
$\chi$ be a compactly supported function on $\M$ that is identically $1$ on the support of $\nabla \phi_i$ for each $i$, then $|\tilde E(k)(z,z')|$ 
is bounded by 
\begin{equation}
\begin{cases}
C \chi(z), \phantom{\ang{d(z_i^\circ,z')}^{-(n_i-1)} \exp(-{kd(z_i^\circ,z')}), \quad} \text{if } z' \in K  \\
C \chi(z) \ang{d(z_i^\circ,z')}^{-(n_i-1)} \exp(-{kd(z_i^\circ,z')}), \quad \text{ if } z' \in \R^{n_i} \times \mathcal{M}_i. 
\end{cases}
\label{tilde E}\end{equation} 


\section{Correcting the low energy parametrix to the true resolvent}
This procedure follows standard lines. We first perturb $\tilde G(k)$ by a finite rank operator so that $\Id + \tilde E(k)$ is perturbed to an invertible operator. Then we analyze the decay properties of the Schwartz kernel of its inverse, and finally compose with the parametrix to obtain the true resolvent. 

\subsection{Finite rank correction}
Estimate \eqref{tilde E} for  $\tilde E(k)$ shows that its Schwartz kernel  is in $L^2(\mathcal{M} \times \mathcal{M})$ uniformly as $k \to 0$. Thus $\tilde E(k)$ is a family of Hilbert-Schmidt operators, and in particular compact. Therefore, $\Id + \tilde E(k)$ is invertible if and only if it has trivial null space. Moreover, the kernel converges as $k \to 0$ pointwise for each $(z,z')$. It follows from this, the uniform bound \eqref{tilde E}, and the dominated convergence theorem, that $E(k)$ is continuous in Hilbert-Schmidt norm as $k \to 0$. 
We add a finite rank correction term to $\tilde G$ to trivialize this null space, so that the operator $\Id + \tilde E(k)$ becomes invertible at $k=0$ and thus, thanks to the continuity just discussed, for all sufficiently small $k$. 

Let $\omega_1, \dots, \omega_N$ be a basis of the null space of $\tilde E(0)$. Notice that each $\omega_i$ is in $C_c^\infty(\mathcal{M})$, as a consequence of the fact that $\tilde E(0)$ has a smooth kernel that is compactly supported in the left variable $z$. Since $\tilde E(0)$ is compact, the operator $\Id + \tilde E(0)$ has closed range of codimension $N$ equal to the dimension of the null space. We claim that there is an $N$ dimensional subspace $V$ spanned by functions $\Delta \rho_1, \dots, \Delta \rho_N$, where each $\rho_i$ is in $C_c^\infty(\mathcal{M})$, such that $V$ is supplementary to the range of $\Id + \tilde E(0)$. This is an immediate consequence of Lemma~\ref{density} below. Given this, we define $G(k)$ to be $G_1(k) + G_2(k) + G_3(k) + G_4$, where 
\begin{equation}\label{G4defn}
G_4 = \sum_{i=1}^N \rho_i \ang{\omega_i, \cdot}.
\end{equation}
(Notice that $G_4$ is in fact independent of $k$.) We also define 
$$
E(k) = (\Delta + k^2) G(k) - \Id.
$$
Thus 
$$
E(k)(z,z') = \tilde E(k)(z,z') + (\Delta + k^2) G_4(z,z') = \tilde E(k)(z,z') + \sum_{i=1}^N \big( (\Delta +k^2) \rho_i(z) \big)\overline{\omega_i(z')}. 
$$
By construction, we have arranged that $\Id + E(0)$ is invertible on $L^2(\mathcal{M})$, and therefore, for sufficiently small $k$, $\Id + E(k)$ is invertible. Let $k_0 > 0$ be such that $\Id + E(k)$ is invertible for $k \leq k_0$. Notice that, after possibly redefining $\chi(z)$ to have larger, but still compact, support, the operator $E(k)$ satisfies pointwise estimates \eqref{tilde E}. 


\begin{lemma}\label{density}
Let $\mathcal{M}$ be as above, and $\Delta$ the Laplacian on $\mathcal{M}$. Then the range of $\Delta$ acting on $C_c^\infty(\mathcal{M})$ is dense in $L^2(\mathcal{M})$.
\end{lemma}

\begin{proof}
Let $f \in L^2(\mathcal{M})$ be a function that is orthogonal to the range of $\Delta$ acting on $C_c^\infty(\mathcal{M})$. We must show that $f$ is identically zero. 

Such a function $f$ is in the domain of the adjoint $\Delta^*$ of the operator $\Delta$ acting on $C_c^\infty(\mathcal{M})$, and satisfying $\Delta^* f = 0$. Since $\Delta$ is formally self-adjoint, this implies that $f$ satisfies $\Delta f = 0$ in the distributional sense, and then, by elliptic regularity, that $f$ is $C^\infty$ and satisfies $\Delta f = 0$ in the classical sense. 
Note that by \cite[Corollary 4.9]{GS}, for $t>1$, 
\begin{equation}\label{ell}
\|\exp({t\Delta})\|_{1\to \infty} \le Ct^{-n_{min}/2}. 
\end{equation}
On the other hand, by Proposition~\ref{prop:CGT}, for $k = [N/4]+1$, the operator  $(\Delta+1)^{-k}$ maps $L^2(\mathcal{M})$ into $L^\infty(\mathcal{M})$. It follows that  $f\in  L^\infty(\mathcal{M})$ and $\exp({t\Delta})f=f$. Hence 
if $f \neq 0$, we have $\|\exp({t\Delta})\|_{1\to \infty} \ge c > 0$. This contradicts estimate  \eqref{ell} for large $t$. 

\end{proof}

\subsection{Inverting $\Id + E(k)$}
Now that we have constructed a parametrix $G(k)$ such that the error term $E(k)$ is such that $\Id + E(k)$ is invertible, for $k \leq k_0$, we may define $S(k)$ such that 
$$
\big( \Id + E(k) \big)^{-1} = \Id + S(k), \quad k \leq k_0. 
$$
For the remainder of this section, we assume that $k \leq k_0$. 
Our goal in this subsection is to show that $S(k)$ has a Schwartz kernel obeying estimates \eqref{tilde E}. 

To do this, we express $\Id = (\Id + E(k))(\Id + S(k)) = (\Id + S(k))(\Id + E(k))$ to conclude that 
$$
S(k) = - E(k) - E(k)S(k) = - E(k) - S(k) E(k).
$$
In particular this shows that $S(k)$ is also Hilbert-Schmidt, with uniformly bounded Hilbert-Schmidt norm. Substituting one expression into the other we find that
$$
S(k) = -E(k) + E(k)^2 + E(k) S(k) E(k).
$$
Next for $a=1$ or $2$, we define weight functions $\omega_a : M \times [0, k_0] \to (0,\infty)$ by 
\begin{equation}
\omega_a(z,k) = \begin{cases} 1, \quad z \in K \\
\ang{d(z_i^\circ,z)}^{-(n_i -a)} \exp(-{kd(z_i^\circ,z)}), \quad z \text{ in the $i$th end}.\end{cases}
\end{equation}
Using these weights we define weighted $L^\infty$ spaces, depending parametrically on $k$, such that 
\begin{equation*}
L^\infty_{\omega_a}(\M) = \{ f \in L^\infty(\M) \mid  \exists C \text{ such that } |f(z)| \leq C \omega_a(z,k) \text{ a.e.} \}.
\end{equation*}
Because we have shown that $E(k)$ satisfies pointwise estimates \eqref{tilde E}, we can say that $E(k)(z,z')$ is a uniformly bounded family of $L^2(\M)_z$ functions with values in the weighted $L^\infty$ space  $ L^\infty_{\omega_1}(\M)_{z'}$, or alternatively, has the form $\chi(z)$ times a uniformly bounded family of $L^\infty(\M)_z$ functions with values in $L^2(\M)_{z'}$. With these descriptions, we see that both 
$$
E(k)^2 = \int_M E(k)(z,z'') E(k)(z'', z) \, dg(z'')$$
and
$$
E(k) S(k) E(k) = \int_M \int_M E(k)(z,z'') S(k)(z'', z''') E(k)(z''', z') \, dg(z'') \, dg(z''') 
$$
are uniformly bounded (in $k$) in the space $$L^\infty_{\chi \otimes \omega_1}(\M^2) := \chi(z) L^\infty(\M; L^\infty_{\omega_1}(\M)).$$ That is, $S$ satisfies \eqref{tilde E}.

\subsection{Correction term}
The exact resolvent is $(\Delta + k^2)^{-1} = G(k) + G(k)S(k)$. So it remains to determine the nature of $G(k) S(k)$. Since $S(k)(z,z')$ is supported in the region $\{ z \in \supp \chi \}$, only points $(z,z')$ where $z' \in \supp \chi$ are relevant for the kernel $G(k)$. We also consider the kernel $\nabla (\Delta + k^2)^{-1}$, for which the correction term is  $\nabla G(k)S(k)$. 

First consider $G_1(k) S(k)$. This can be expressed as 
$$
 \sum_{i=1}^l\phi_i (z) (\Delta_{R^{n_i} \times \mathcal{M}_i} + k^2)^{-1} \Big( \phi_i  S(k)(\cdot, z')\Big)(z),
$$
which using \eqref{re-sup} has the form 
\begin{equation}
G_1(k) S(k) \in  
L^\infty_{\omega_2 \otimes \omega_1}(\M^2),
\label{G1S}\end{equation}
in the sense that for $z$ on the $i$th end and $z'$ on the $j$th end, the kernel is bounded by 
$$
C\ang{d(z_i^\circ,z)}^{-(n_i-2)} \exp(-{kd(z_i^\circ,z)})  \ang{d(z_j^\circ,z')}^{-(n_j-1)} \exp(-{kd(z_j^\circ,z')}) .
$$
If we apply a spatial derivative to $G_1(k)$, then using \eqref{g-res} we find that 
the kernel $\nabla  G_1(k) S(k)$ has the form 
\begin{eqnarray*}
\nabla G_1(k) S(k)& \in L^\infty_{\omega_1 \otimes \omega_1}(\M^2),
\label{nablaG1S}\end{eqnarray*}
which is similar to the form of $ G_1(k) S(k)$ above, but with exponent $n_i-2$ 
replaced by $n_i-1$. 
The kernel $G_2(k)S(k)$ is simpler as $G_2(z,z')$ has compact support in $z$. Hence $G_2(k) S(k)$ and $ \nabla G_2 S(k)$ satisfies \eqref{tilde E}. In other terms, for some compactly supported function $\chi$  we have 
\begin{equation*}
G_2(k) S(k), \nabla G_2 S(k)  \in L^\infty_{\chi \otimes \omega_1}(\M^2).
\label{G2S}\end{equation*}
Next, we have, thanks to \eqref{re-sup}, 
\begin{equation*}
G_3(k) S(k) \in L^\infty_{\omega_2 \otimes \omega_1}(\M^2)
\label{G3S}\end{equation*}
and, for the same reasons as in the case of $\nabla G_1(k) S(k)$, 
\begin{equation*}
\nabla G_3(k) S(k) \in L^\infty_{\omega_1 \otimes \omega_1}(\M^2).
\label{nablaG3S}\end{equation*}
Finally, 
\begin{equation*}
G_4(k) S(k), \nabla G_4 S(k)  \in  L^\infty_{\chi \otimes \omega_1}(\M^2) .
\label{G4S}\end{equation*}
In summary, we have shown
\begin{proposition}\label{prop:leres}
For $k \leq k_0$, the resolvent $(\Delta + k^2)^{-1}$ takes the form 
$$
(\Delta + k^2)^{-1}  = G(k) + G(k) S(k),$$  where 
$G(k) = G_1(k) + G_2(k) + G_3(k) + G_4(k)$ is defined in \eqref{G(k)} and \eqref{G4defn}, while $G(k) S(k) \in L^\infty_{\omega_2 \otimes \omega_1}(\M^2)$, that is, $G(k) S(k)$ has a pointwise kernel estimate
\begin{equation}
\Big| \big( G(k) S(k) \big)(z,z')\Big|  \leq C \omega_2(z,k) \omega_1(z',k). 
\label{GS}\end{equation}
Similarly, $\nabla (\Delta + k^2)^{-1}$ takes the form $\nabla G(k) + \nabla G(k) S(k)$, where
\begin{equation}
\Big| \big( \nabla G(k) S(k) \big)(z,z')\Big|  \leq C \omega_1(z,k) \omega_1(z',k). 
\label{nablaGS}\end{equation}
\end{proposition}

\

\subsection{Significance of the $G_3$ term} We now make some more detailed comments on why the $G_3$ term is included in the parametrix. 

One reason is that, for $n_i = 3$ or $4$ (recall we have assumed each $n_i \geq 3$), the error term $E(0)$ would not be in $L^2$ unless we included the $G_3$ term,
as the error would only decay as $\ang{d(z_i^\circ,z')}^{n_i - 2}$ as $d(z_i^\circ,z') \to \infty$. However, the more important reason for including $G_3$ is that the error term $G(k)S(k)$ then decays to order $\ang{d(z_i^\circ,z')}^{n_i-1}$ as $d(z_i^\circ,z') \to \infty$ (see \eqref{GS}). This decay is faster than the decay of the $G_3(k)$ term, which decays as $\ang{x'_id(z_i^\circ,z')}^{n_i-2}$ as $d(z_i^\circ,z') \to \infty$. Therefore, \emph{$G_3(k)$ gives the leading behaviour of the true resolvent kernel in this asymptotic regime} (where $d(z_i^\circ,z') \to \infty$ while $z_i$ remains in a fixed but arbitrary compact set). Moreover, we shall see in Section~\ref{sec7} that the range of $p$ for which the Riesz transform is bounded on $L^p$ is governed by the asymptotics of the resolvent in exactly this regime. So the $G_3(k)$ term --- which is the one \emph{not} of Calder\'on-Zygmund type ---  is key to determining the boundedness of the Riesz transform. These observations were already present in \cite{CCH}. 

%


\section{Riesz transform localized to low energies}
In the previous section, we constructed the resolvent kernel $(\Delta + k^2)^{-1}$ for $k \leq k_0$. In this section we shall analyze the boundedness on $L^p(\M)$ of the operator 
\begin{equation*}
\nabla F_<(\sqrt{\Delta})=\frac{2}{\pi} \nabla \int_0^{k_0} (\Delta + k^2)^{-1} \, dk
\end{equation*}
which we call the Riesz transform localized to low energies, see \eqref{F<>}. 

\begin{proposition}\label{prop:le-Riesz}
The Riesz transform localized to low energies, $\nabla F_<(\sqrt{\Delta})$, is of weak type $(1,1)$ and  bounded on $L^p(\M)$ for $p$ in the range $(1, \min_i n_i)$. 
\end{proposition}

\begin{proof}
We decompose $\nabla F_<(\sqrt{\Delta})$ by expressing $(\Delta + k^2)^{-1} = G_1(k) + G_2(k) + G_3(k) + G_4(k) + G(k) S(k)$ as in Proposition~\ref{prop:leres}, and treating each separately. 

$\bullet \ G_1$ term. Here we need to analyze the boundedness of 
\begin{equation}
\frac{2}{\pi}  \int_0^{k_0} \nabla \Big( (\Delta_{\R^{n_i} \times \mathcal{M}_i} + k^2)^{-1}(z, z') \phi_i(z) \phi_i(z') \Big) \, dk
\end{equation}
which we can view as an operator on $\R^{n_i} \times \mathcal{M}_i$. We break this kernel into two pieces, according to whether the derivative $\nabla$ hits the $\phi(z)$ factor or the resolvent factor. We first consider the term that results when the derivative hits the $\phi$ factor.

%

Next set $D_r=\{(z,z')\in \M^2\colon \, d(z,z') \le r \}$ and let $\chi_{D_r}$ be the characteristic function of the set $D_r$. 
Then it follows from \eqref{re-sup} that 

\begin{equation}\label{eq36}
\|\int_0^{k_0}\chi_{D_r}(\Delta_{\R^{n_i} \times \mathcal{M}_i} + k^2)^{-1}\|_{p\to p }\le 
C_r
\end{equation}
for all $p$. 

If we write $R_k(z,z')$ for the kernel of the operator $(1-\chi_{D_r})(\Delta_{\R^{n_i} \times \mathcal{M}_i} + k^2)^{-1}$, then provided $q < n/(n-2)$, we obtain from \eqref{re-sup} 
$$
\| R_k(z,z') \|_{L^\infty(z); L^q(z')} \leq C k^{-2 + n_i(1 - 1/q)}, \quad \| R_k(z,z') \|_{L^q(z); L^\infty(z')} \leq C k^{-2 + n_i(1 - 1/q)}.
$$
This immediately implies that this operator is bounded as a map from $L^{q'} \to L^\infty$ and from $L^1 \to L^q$ with operator norm bounded by $C k^{-2 + n_i(1 - 1/q)}$. 
Interpolating, we find that 
\begin{equation*}
\|(1-\chi_{D_r})(\Delta_{\R^{n_i} \times \mathcal{M}_i} + k^2)^{-1}\|_{p\to q }\le 
C_r k^{n_i(1/p-1/q)-2}
\end{equation*}
for all   $p<q$ such that $1/p-1/q<2/n_i$. Hence 
\begin{eqnarray*}
\Big\|\nabla \phi_i\int_0^{k_0}(1-\chi_{D_r})(\Delta_{\R^{n_i} \times \mathcal{M}_i} + k^2)^{-1} \, dk \Big\|_{p\to p }\\ \le C
\Big\|\int_0^{k_0}(1-\chi_{D_r})(\Delta_{\R^{n_i} \times \mathcal{M}_i} + k^2)^{-1} \, dk \Big\|_{p\to q }\\
\le C \int_0^{k_0}k^{n_i(1/p-1/q)-2} dk \le C < \infty
\end{eqnarray*}
for all $p<q$ such that $2/n_i>1/p-1/q>1/n_i$. Together with \eqref{eq36} (and recalling that $\nabla \phi_i$ is compactly supported) 
this implies that 
\begin{eqnarray*}
	\|\nabla \phi_i\int_0^{k_0}(\Delta_{\R^{n_i} \times \mathcal{M}_i} + k^2)^{-1}\|_{p\to p } \le C < \infty
\end{eqnarray*}
for all  $p< n_i$. 


Next consider the term that results when the derivative hits the resolvent factor. We factorise this operator into composition of multiplication 
by $\phi_i$; the Riesz transform $\nabla (\Delta_{\R^{n_i} \times \mathcal{M}_i})^{-1/2}$; and another multiplication by $\phi_i$. Then by 
Lemma~\ref{l2.1}
\begin{eqnarray*}
\Big\| \phi_i\int_0^{k_0}\nabla(\Delta_{\R^{n_i} \times \mathcal{M}_i} + k^2)^{-1}dk  \,  \phi_i\Big\|_{p\to p }\hspace{3cm}\\	=
\Big\| \phi_i\nabla F_<(\Delta_{\R^{n_i} \times \mathcal{M}_i})  \phi_i\Big\|_{p\to p }\hspace{3cm}\\ = 
\Big\|\phi_i  \nabla (\Delta_{\R^{n_i} \times \mathcal{M}_i})^{-1/2} \Big( \frac{\pi}{2} -  \tan^{-1}\Big(\frac{\sqrt{\Delta_{\R^{n_i} \times \mathcal{M}_i}}}{k_0}\Big) \Big)\phi_i\Big\|_{p\to p}\\
\le C
\Big\| \nabla (\Delta_{\R^{n_i} \times \mathcal{M}_i})^{-1/2}  \Big\|_{p\to p}
\end{eqnarray*}
and the last operator norm is finite, as follows from standard results on Riesz transforms (it can be derived easily from 
 from gradient bounds of the heat kernel, as in \eqref{g}). 
Note that multiplier $\pi/2 - \tan^{-1}\Big(\frac{\sqrt{\Delta_{\R^{n_i} \times \mathcal{M}_i}}}{k_0}\Big)$ is also bounded on $L^1$ so the whole term is of weak 
type $(1,1)$ as well.

$\bullet \ G_2$ term. The operator $\nabla G_2(k)$ is a family of pseudodifferential operators of order $-1$, with Schwartz kernel having compact support, depending smoothly on $k$. Therefore the integral is a pseudodifferential operator of order $-1$ with Schwartz kernel having compact support. It is therefore bounded on $L^p$ spaces for all $p \in [1, \infty]$. 

$\bullet \ G_3$ term. This term is smooth so we only have to analyze the decay of the kernel as $z$ or $z'$ tends to infinity. 
If in the expression for $G_3(k)$ given in \eqref{G(k)} we bound $(\Delta_{\R^{n_i} \times \mathcal{M}_i} + k^2)^{-1}(z_i^\circ,z')$
 by \eqref{re-sup} and $\nabla u_i$ by \eqref{u-estimates}, then we can integrate in $k$ to obtain a bound on this term. When $z'$ lies in a compact set and $z$ goes to infinity along the $i$th end, we get a bound
$$
\ang{d(z_i^\circ,z)}^{-(n_i-1)} \int_0^{k_0}  \exp(-{kd(z_i^\circ,z)})\, dk \leq C  \ang{d(z_i^\circ,z)}^{-(n_i)}.
$$
When $z$ lies in a compact set and $z'$ goes to infinity along the $j$th end we obtain a bound
$$
\ang{d(z_i^\circ,z)}^{-(n_i-2)} \int_0^{k_0} \int_0^{k_0}  \exp(-{kd(z_i^\circ,z)}) \, dk \leq C  \ang{d(z_i^\circ,z)}^{-(n_i-1)}
$$
and when $z$ goes to infinity along the $i$th end and $z'$ goes to infinity along the $j$th end\footnote{The case $i=j$ is also included.} we have a bound 
\begin{eqnarray}\nonumber 
\frac1{\ang{d_i(z_i^\circ,z)}^{n_i - 1} \ang{d_j(z_j^\circ,z')}^{n_j - 2}} \int_0^{k_0} \exp(-k(d_i(z_i^\circ, z)+d_j(z_j^\circ, z')))  \, dk 
\hspace{1cm} 
\\
\hspace{1cm} \leq C \min \Big( \frac1{\ang{d_i(z_i^\circ,z)}^{n_i } \ang{d_j(z_j^\circ,z')}^{n_j - 2}}  ,  \frac1{\ang{d_i(z_i^\circ,z)}^{n_i-1 } \ang{d_j(z_j^\circ,z')}^{n_j - 1}}\Big).
\label{ijbound}\end{eqnarray}
In each case we see that the kernel is in $L^p(\M; L^{p'}(\M))$ for $1 < p < \min_i n_i$, so we have boundedness on $L^p(\M)$ for $p$ in this range. Note also that  for any fixed $z\in \R^{n_i} \times \mathcal{M}_i$ on the $i$th end, the RHS of \eqref{ijbound} shows that the $L^\infty(\M)$ norm with respect to $z'$
is uniformly bounded by $C \ang{d_i(z_i^\circ,z)}^{-n_i}$. It follows that this kernel maps $L^1(\M)$ to to an $L^\infty$ function decaying as $C \ang{d_i(z_i^\circ,z)}^{-n_i}$ along each end, which is clearly in weak $L^1$, so the corresponding operator is also of weak type $(1,1)$. 

$\bullet \ G_4$ term. This term contributes the following to the low energy Riesz transform:
\begin{equation}
k_0 \sum_{i=1}^N \nabla \rho_i \ang{\omega_i, \cdot}.
\end{equation}
Since both $\rho_i$ and $\omega_i$ are in $C_c^\infty(\M)$, this operator is bounded on $L^p$ for all $p \in [1, \infty]$.

$\bullet \ GS$ term. This term can be treated in the same way as the $G_3$ term, with the difference that it vanishes to an additional order in the right (primed) variable. We arrive at the bound 
$$
C \frac1{\ang{d_i(z_i^\circ,z)}^{n_i - 1}}
$$
when $z'$ lies in a compact set and $z$ goes to infinity along the $i$th end, the bound 
$$
C  \frac1{\ang{d_j(z_j^\circ,z')}^{n_j}} 
$$
when $z$ lies in a compact set and $z'$ goes to infinity along the $j$th end, and the bound  
\begin{equation}
C \min \Big( \frac1{\ang{d_i(z_i^\circ,z)}^{n_i}\ang{d_j(z_j^\circ,z')}^{n_j-1}}\, , \,
\frac1{\ang{d_i(z_i^\circ,z)}^{n_i-1}\ang{d_j(z_j^\circ,z')}^{n_j}}\Big)
\label{GSijbound}\end{equation}
and when $z$ goes to infinity along the $i$th end and $z'$ goes to infinity along the $j$th end. 
In each case we see that the kernel is in $L^p(\M; L^{p'}(\M))$ for $1 < p < \infty$, so we have boundedness on $L^p(\M)$ for all $p \in (1, \infty)$. The argument which we used above to verify weak type $(1,1)$ also remains valid. 

This completes the proof of Proposition~\ref{prop:le-Riesz}. 
\end{proof}


\section{Riesz transform localized to high energies}\label{sec:high}
Recall that 
$$
\Delta^{-1/2}=\frac{2}{\pi}\int_0^\infty (\Delta+k^2)^{-1} dk = F_<(\sqrt{\Delta}) + F_>(\sqrt{\Delta})
$$
where $F_<$ and $F_>$ are defined in \eqref{F<>}. In view of Proposition~\ref{prop:le-Riesz}, to prove boundedness of the Riesz transform for $p < \min_i n_i$, it suffices to prove 

\begin{proposition}\label{prop:he-Riesz}
The Riesz transform localized to high energies, $\nabla F_>(\sqrt{\Delta})$, is bounded on $L^p(\M)$ for $p$ in the range $(1, \infty)$. 
In addition, it is of weak-type $(1,1)$, that is, it is a bounded map from $L^1(\M)$ to $L^1_w(\M)$. 
\end{proposition}

\begin{proof}
We shall show that the kernel of the operator $\nabla F_>(\Delta)$ is integrable away from the diagonal 
and that close to the diagonal it satisfies the classical Calder\'on-Zygmund condition. 

We begin by noting that the Fourier transform of $F_>$  has exponential decay: 
 $$
|\hat F_> (t)| \le Ce^{-k_0|t|} 
$$
for all $|t| >1$. 
Indeed, this follows by writing
$$
\hat F_> (t)  = \frac 2{\pi}\int_{-\infty}^\infty \int_{k_0}^\infty e^{-it\lambda} \big( \lambda^2 + k^2 \big)^{-1} \, dk \, d\lambda
= \int_{k_0}^\infty \frac{e^{-k |t|}}{k} \, dk.
$$


Next let $s \in C_c^{\infty}(\R)$ be an even  compactly supported function such that 
$0 \le s (r)\le 1$,  $s(r) =1$ for all $-1/2\le k \le 1$ and $s(r)=0$ for $|r| \ge 1$.
 Then for any $r>0$ we define functions $G'_r$ and $G''_r$ 
in the following way. 
We define $G'_r$ as the inverse Fourier transform of $\hat{G_r}$ where 
\begin{eqnarray*}
\hat{G'_r}(t) =\hat{F_>}(t) s \Big( \frac{t}{r} \Big).
\end{eqnarray*}
Thus the Fourier transform of $G'_r$ is contained in the interval $[-r, r]$. 
Then $G''_r$ is defined by the relation $G''_r(\lambda)+G'_r(\lambda)=F_>(\lambda)$. That is, 
\begin{eqnarray}
\hat{G''_r}(t) =\hat{F_>}(t) \Big(1 - s \Big( \frac{t}{r} \Big) \Big).
\label{G''}\end{eqnarray}

Clearly, $\hat{G''_r}(t)$ is a Schwartz function of $t$, and therefore $G''_r(\lambda)$ is a Schwartz function of $\lambda$. This in turn implies, using \eqref{symbol}, that 
$G'_r(\lambda)$ is a symbol of order $-1$ in $\lambda$. 

We can also see from \eqref{G''} that the $L^1$ norm of $\hat G''_r$ is bounded by $Ce^{-k_0 r}$ for $r \geq 1$, and, consequently, we have
\begin{equation}
\sup_\lambda |G''_r(\lambda)| \leq C e^{-k_0 r}, \quad r \geq 1. 
\label{G''norm}\end{equation}

The reason for introducing this decomposition is that we can express $G'_r(\sqrt{\Delta})$ and $G''_r(\sqrt{\Delta})$ in terms of the cosine wave kernel, 
and exploit the finite propagation speed of the cosine wave kernel. Indeed, we have (using the evenness of $\hat F_>(t)$)
\begin{multline}
G'_r(\sqrt{\Delta}) = \frac1{2\pi} \int_{-\infty}^\infty e^{i \sqrt{\Delta} t} \hat{F_>}(t) s \Big( \frac{t}{r} \Big) \, dt 
= \frac1{\pi} \int_{0}^\infty \cos(\sqrt{\Delta} t) \hat{F_>}(t) s \Big( \frac{t}{r} \Big) \, dt .
\label{cosine}\end{multline}
It follows immediately from \eqref{cosine} and from the finite speed of propagation of $\cos(\sqrt{\Delta} t)$ that $G'_r(\sqrt{\Delta})$ has Schwartz kernel supported in the set $ \{ (z, z') \mid d(z,z') \leq r \}$. Now choose $r = r_* > 0$ to be half the injectivity radius of $\M$. 
Finite speed of propagation means that $G'_{r_*}(\sqrt{\Delta_\M})(\cdot, y)$ is identical to the kernel of $G'_{r_*}(\sqrt{\Delta_\mathcal{N}})(\cdot, y)$ for any Riemannian manifold $\mathcal{N}$ that is isometric to $\M$ in the ball of radius $r_*$ about $y$. By our choice of $r_*$, this ball is contractible, so for $y \in K$, we can take $\mathcal{N}$ to be a sphere with a Riemannian metric such that the ball of radius $r_*$ about the south pole is isometric to $B(y, r_*)$ in $\M$.  On the other hand, for $y \notin K$, then for $y$  belonging to the $i$th end, we can take $\mathcal{N}$ to be $2r_*\T^{n_i} \times \mathcal{M}_i$, where $\T^n = \R^n/\Z^n$. For a compact manifold $\mathcal{N}$, $G'_{r_*}(\sqrt{\Delta_\mathcal{N}})$ is a pseudodifferential operator of order $-1$, since the function $G_r(\lambda)$ is a symbol of order $-1$ --- see for example \cite{Dunau}, \cite[Chapter XII, Section 1]{Taylor}, or \cite{HV}. It follows that $G'_{r_*}(\sqrt{\Delta_\M})$ is a pseudodifferential operator of order $-1$ in a uniform sense (since we only need take a compact set of $\mathcal{N}$'s as explained above). Therefore, $K_{\nabla G'_{r_*}(\sqrt{\Delta}))}$ is weak type $(1,1)$ and 
bounded on  $L^p(\M)$ for all $1<p< \infty$ by the standard Calder\'on-Zygmund argument.


Now we turn to the double-primed operator $G''_{r_*}(\sqrt{\Delta})$. We shall apply Schur's test, and
show that there exists a constant $C$ such that 
\begin{eqnarray}\label{x}
\sup_y\int_{\M} \big| K_{\nabla G''_{r_*}(\sqrt{\Delta})}(x,y) \big|dx \le C
\end{eqnarray}
and 
\begin{eqnarray}\label{y}
\sup_x \int_{\M }\big|K_{\nabla G''_{r_*} (\sqrt{\Delta})}(x,y) \big| dy \le C
\end{eqnarray}
which implies boundedness on all $L^p$ spaces, $1 \leq p \leq \infty$. 

To prove \eqref{x} we shall show that 
\begin{equation}
\sup_y\int_{x\notin B(y,r)} \big| K_{\nabla G''_{r_*}(\sqrt{\Delta})}(x,y)\big| ^2 dx \le C e^{-k_0r/2}, \quad r \geq r_*.
\label{squared}\end{equation}
This suffices since it implies in particular that 
$$
\sup_y\int_{r \leq d(x,y) \leq 2r}\big| K_{\nabla G''_{r_*}(\sqrt{\Delta})}(x,y)\big| ^2 dx \le C e^{-k_0r/2} , \quad r \geq r_*.
$$
The measure of the set $\{x \in M \mid r \leq d(x,y) \leq 2r \}$ is bounded by $C r^N$, $r \geq r_*$, where $N$ is the dimension of $\M$, uniformly in $y \in \M$.
So we can apply H\"older's inequality to find that 
$$
\sup_y\int_{r \leq d(x,y) \leq 2r}|K_{\nabla G''_{r_*}(\sqrt{\Delta})}(x,y)| dx \le C e^{-k_0r/4}.
$$
These estimates can then be summed over a sequence of dyadic annuli to obtain \eqref{x}. (The exponential decay in \eqref{squared} is clearly more than we need; the argument only requires that we have polynomial decay that beats the polynomial volume growth of $\M$.) 

Now a key observation is that, due to the support properties of $s$, finite propagation speed, and the identity 
\begin{equation}
G''_r(\sqrt{\Delta}) = \frac1{\pi} \int_{0}^\infty \cos(\sqrt{\Delta} t) \hat{F_>}(t) \Big( 1 - s \Big( \frac{t}{r} \Big) \Big) \, dt ,
\label{cosine''}\end{equation}
we have 
$$
K_{\nabla G''_{r_*}(\sqrt{\Delta})}(x,y)=K_{\nabla G''_r(\sqrt{\Delta})}(x,y) \text{ if } x\notin B(y,r), \quad r \geq r_* .
$$
 Hence 
\begin{equation}\begin{aligned}
\int_{x\notin B(y,r)}|K_{\nabla G''_{r_*}(\sqrt{\Delta})}(x,y)|^2dx &=
\int_{x\notin B(y,r)}|K_{\nabla G''_r(\sqrt{\Delta})}(x,y)|^2dx \\ &\le
 \int_{x \in M}|K_{\nabla G''_r(\sqrt{\Delta})}(x,y)|^2dx  \\
 &= \Big\langle \nabla_x K_{G''_r(\sqrt{\Delta})}(\cdot,y), \nabla _x K_{G''_r(\sqrt{\Delta})}(\cdot,y) \Big\rangle \\
 &= \Big\langle \Delta K_{G''_r(\sqrt{\Delta})}(\cdot,y), K_{G''_r(\sqrt{\Delta})}(\cdot,y) \Big\rangle \\
 &= \Big\langle \Delta^{1/2} K_{G''_r(\sqrt{\Delta})}(\cdot,y), \Delta^{1/2} K_{G''_r(\sqrt{\Delta})}(\cdot,y) \Big\rangle \\
 &= \int_{x \in M}|K_{ \sqrt{\Delta} G''_r(\sqrt{\Delta})}(x,y)|^2dx ,
 \end{aligned}\end{equation}
compare with  formulae  (5.3) and (6.6) of \cite{Si}. This term can be estimated by 
\begin{eqnarray*}
 \sup_y \int_{x \in M}|K_{\sqrt{\Delta} G''_r(\sqrt{\Delta})}(x,y)|^2dx &=&
 \big\|\sqrt{\Delta} G''_r(\sqrt{\Delta})\big\|^2_{1\to 2} =  \big\|\sqrt{\Delta} G''_r(\sqrt{\Delta})\big\|^2_{2 \to \infty}\\
 &\le&  \big\|  (I+\Delta)^{ n}  \sqrt{\Delta} G''_r(\sqrt{\Delta}) \big\|^2_{2  \to 2}   \big\|(I+\Delta)^{-n}\big\|^2_{2 \to \infty} \\[6pt]
 &\leq& C \sup_{\lambda \geq 0} 
 \big|(1+\lambda^2)^n {\lambda}G''_r(\lambda)\big|^2 \\
 &\le& C e^{-k_0 r/2}. 
\end{eqnarray*}
Here we employed Proposition~\ref{prop:CGT} for the $L^2 \to L^\infty$ operator norm of the operator $(I + \Delta)^{-n}$. 
We also used \eqref{G''norm} in the last line. 
This proves estimate \eqref{x} and shows that $\|\nabla G''_{r_*}(\sqrt{\Delta})\|_{1\to1} \le C$. 

To obtain \eqref{y}, we use the Hodge Laplacian $\tilde \Delta$ on differential forms. Recall that the exterior derivative
$d$ and the metric induces a dual operator $d^*$ mapping $q$-forms to $q-1$-forms, and the Hodge Laplacian
is defined by $\tilde \Delta = d d^* + d^* d$. It commutes with both $d$ and $d^*$. As a consequence, any function of $\tilde \Delta$
commutes with both $d$ and $d^*$. In particular, $d G''_r(\sqrt{\tilde \Delta}) =  G''_r(\sqrt{\tilde \Delta}) d$; if we write $\Delta_q$ for the
Hodge Laplacian acting on $q$-forms then we can write $d G''_r(\sqrt{ \Delta_0}) =  G''_r(\sqrt{ \Delta_1}) d$. 
So, in \eqref{y}, if we write $d G_r''(\sqrt{ \Delta_0})$ for this operator (thinking of the gradient $\nabla$ as the composition $I \circ d$ where $I$ is the identification of 1-forms and vector fields using the metric tensor), then this operator is the same as $G_r''(\sqrt{ \Delta_1}) d$. Moreover, since
$\cos t \sqrt{\Delta_1} $ still satisfies finite speed propagation, we still have
$$
K_{G''_{r_*}(\sqrt{\Delta_1})d}(x,y)=K_{G''_r(\sqrt{\Delta_1})d}(x,y)
$$
when $x \notin B(y, r)$ and $r \geq r_*$. Also note that the Schwartz kernel $K_{G''_r(\sqrt{\Delta_1})d}(x, y)$ of the operator $G''_r(\sqrt{\Delta_1})$ is equal to 
$d^*_y K_{G''_r(\sqrt{\Delta_1})}(x, y)$ as follows by integrating by parts. This allows us to run
the previous argument with one extra step:
\begin{equation}\begin{aligned}
\int_{x\notin B(y,r)}|K_{G''_{r_*}(\sqrt{\Delta_1})d}(x,y)|^2dy &=
\int_{x\notin B(y,r)}|K_{ G''_r(\sqrt{\Delta_1})d}(x,y)|^2dy \\ &\le
 \int_{x \in M}|K_{G''_r(\sqrt{\Delta_1})d}(x,y)|^2dy  \\
 &= \Big\langle d^*_y K_{G''_r(\sqrt{\Delta_1})}(x, \cdot), d^*_y K_{G''_r(\sqrt{\Delta_1})}(x, \cdot) \Big\rangle \\
 \leq \Big\langle d^*_y K_{G''_r(\sqrt{\Delta_1})}(x, \cdot), d^*_y K_{G''_r(\sqrt{\Delta_1})}(x, \cdot) \Big\rangle &+ \Big\langle d_y K_{G''_r(\sqrt{\Delta_1})}(x, \cdot), d_y K_{G''_r(\sqrt{\Delta_1})}(x, \cdot) \Big\rangle \\
 &= \Big\langle (d d^* + d^* d)_y K_{G''_r(\sqrt{\Delta_1})}(x, \cdot), K_{G''_r(\sqrt{\Delta_1})}(x, \cdot) \Big\rangle \\
  &= \Big\langle \Delta_1 K_{G''_r(\sqrt{\Delta_1})}(x, \cdot), K_{G''_r(\sqrt{\Delta_1})}(x, \cdot) \Big\rangle \\
 &= \Big\langle \Delta_1^{1/2} K_{G''_r(\sqrt{\Delta_1})}(x, \cdot), \Delta_1^{1/2} K_{G''_r(\sqrt{\Delta_1})}(x, \cdot) \Big\rangle \\
& = \int_{x \in M}|K_{ \sqrt{\Delta_1} G''_r(\sqrt{\Delta_1})}(x,y)|^2dx .
 \end{aligned}\end{equation}
 
 We complete the argument as before, using Remark~\ref{rem:forms} for the $L^2 \to L^\infty$ estimate on $(\Id + \Delta_1)^{-n}$ instead of Proposition~\ref{prop:CGT}. 
%
%
%
\begin{eqnarray*}
 \int_{x \in M}|K_{\sqrt{\Delta_1} G''_r(\sqrt{\Delta_1})}(x,y)|^2dx &=&
 \big\|\sqrt{\Delta_1} G''_r(\sqrt{\Delta_1})\big\|^2_{1\to 2} =  \big\|\sqrt{\Delta_1} G''_r(\sqrt{\Delta_1})\big\|^2_{2 \to \infty}\\
 &\le&  \big\|  (I+\Delta_1)^{ n}  \sqrt{\Delta_1} G''_r(\sqrt{\Delta_1}) \big\|^2_{2  \to 2}   \big\|(I+\Delta_1)^{-n}\big\|^2_{2 \to \infty} \\[6pt]
 &\leq& C \sup_{\lambda \geq 0}
 \big|(1+\lambda^2)^n {\lambda}G''_r(\lambda)\big|^2 \\
 &\le& C e^{-k_0 r/2}. 
\end{eqnarray*}
\end{proof}


\section{Unboundedness of the Riesz transform for large $p$}\label{sec7}

We complete the proof of Theorem~\ref{thm:main} by showing that Proposition~\ref{prop:le-Riesz} is sharp in terms of the range of $p$.

\begin{proposition}\label{prop:Riesz-unbdd}
Assume that $\M$ has at least two ends, and let $p$ be any exponent greater than or equal to  $\min_i n_i$. Then the Riesz transform is not bounded on $L^p(\M)$. 
\end{proposition}

\begin{proof}
In view of Proposition~\ref{prop:he-Riesz}, it suffices to show that the low energy Riesz transform is not bounded on $L^p(\M)$. And examining the proof of  
Proposition~\ref{prop:le-Riesz}, it suffices to show that the contribution of the terms arising from $G_1$ and $G_3$ are not bounded. Furthermore, the $G_1$ term naturally divided into the term where the gradient hit the $\phi(z)$ factor and the term where the gradient hit the resolvent factor, and the latter was bounded on $L^p$ for all $p \in (1, \infty)$. 
Therefore, it suffices to show that 
\begin{eqnarray}\label{badterm1}
 \int_0^{k_0} \Bigg( \nabla \phi_i(z) (\Delta_{\R^{n_i} \times \mathcal{M}_i} + k^2)^{-1}(z,z')  \phi_i(z') \hspace{2cm}\\+  {\nabla u_i(z,k) } (\Delta_{\R^{n_i} \times \mathcal{M}_i} + k^2)^{-1}(z_i^\circ,z') \phi_i(z') \Bigg) \, dk\nonumber 
\end{eqnarray}
does not act boundedly on $L^p(\M)$ for $p \geq \min_i n_i$. 

We make a series of simplifications. First, in the resolvent factor above depending on $z$, we may replace $z$ by $z_i^\circ$ as the difference 
 can be estimated by 
\begin{eqnarray}
 \int_0^{k_0}  \nabla \phi_i(z)
  |(\Delta_{\R^{n_i} \times \mathcal{M}_i} + k^2)^{-1}(z_i^\circ,z')
  - (\Delta_{\R^{n_i} \times \mathcal{M}_i} + k^2)^{-1}(z,z')|  \phi_i(z')\\
  \le \int_0^{k_0} d(z_i^\circ,z)(\ang{d(z_i^\circ,z')}^{-(n-1)}
  \exp(-kd(z_i^\circ,z)).
\end{eqnarray}\label{replacebyzicirc}
 In a similar way to the $GS$ term in the proof of Proposition~\ref{prop:le-Riesz},
 after the integration, \eqref{replacebyzicirc} can be estimated by $O(\ang{d(z_i^\circ,z')}^{-n_i})$. 
 Hence the difference  acts as a bounded operator on $L^p$ for all $p \in (1, \infty)$. 
 Using this in \eqref{badterm1} shows that it suffices to prove that 
\begin{equation*}
  \int_0^{k_0}  \Big( \nabla \phi_i(z) + \nabla u_i(z,k) \Big) (\Delta_{\R^{n_i} \times \mathcal{M}_i} + k^2)^{-1}(z_i^\circ,z') \phi_i(z') \, dk
\end{equation*}
does not act boundedly on $L^p(\M)$  for $p \geq \min_i n_i$. 

Choose a nonnegative function $\tau(z)$, compactly supported and not identically zero, supported on one of the ends, say the $j$th end. Clearly, it suffices to show that 
\begin{equation}
\tau(z)  \int_0^{k_0} \Big( \nabla \phi_i(z) + \nabla u_i(z,k) \Big)(\Delta_{\R^{n_i} \times \mathcal{M}_i} + k^2)^{-1}(z_i^\circ,z') \phi_i(z') \, dk
\label{badterm3}\end{equation}
does not act boundedly on $L^p(\M)$  for $p \geq \min_i n_i$. 

We split the integral \eqref{badterm3} into two parts, by writing $\nabla \phi_i(z) + \nabla u_i(z,k) = \nabla (u_i(z,k) - u_i(z,0)) + (\nabla \phi_i(z) + \nabla u_i(z,0))$. We next show that the first part is a bounded operator on $L^p(\M)$ for all $p \in (1, \infty)$. Thanks to \eqref{u-kminus0-grad}, we can write $\nabla (u_i(z,k) - u_i(z,0)) = k {w}(z,k)$, where $ \tau w(z, k)$ is uniformly bounded in $L^p(\M)$ for $k \in [0, k_0]$. 
So it suffices to show that the operator with kernel 
\begin{equation*}
\tau(z)  \int_0^{k_0}  k w(z, k)  (\Delta_{\R^{n_i} \times \mathcal{M}_i} + k^2)^{-1}(z_i^\circ,z') \phi_i(z') \, dk
\end{equation*}
is bounded on $L^p(\M)$ for all $p \in (1, \infty)$. To do this, we show that for each fixed $k$, the integral operator above is bounded on $L^p(\M)$, and the operator norm, as a function of $k$, is integrable on the interval $[0, k_0]$. Notice that for each fixed $k$, we have  a kernel of the form $a(z) \ang{b(z), \cdot}$; that is, a rank one operator. It is bounded on $L^p$ if and only if $a \in L^p$ and $b \in L^{p'}$, and then the operator norm is $\| a \|_{L^p(\M)} \| b \|_{L^{p'}(\M)}$. We take $a = w(\cdot, k)$, clearly uniformly bounded in $L^p(\M)$. Next let $R$ be the distance form $z_i^\circ$ to the support of the 
function $\phi_i$.  Then using estimates \eqref{re-sup} we compute the $L^{p'}$ norm of $b$:
\begin{equation}\begin{gathered}
\| b \|_{L^{p'}(M)} = {k} \Bigg( \int_{\R^{n_i} \times \mathcal{M}_i}  \Big| (\Delta_{\R^{n_i} \times \mathcal{M}_i} + k^2)^{-1}(z_i^\circ,z') \phi_i(z') \Big|^{p'}   \, dz'  \Bigg)^{1/p'} \\
\leq C k \Bigg( \int_{d(z_i^\circ, z')| \geq R}  \Big| (\Delta_{\R^{n_i} \times \mathcal{M}_i} + k^2)^{-1}(z_i^\circ,z')  \Big|^{p'}   \, dz ' \Bigg)^{1/p'} \\
\le Ck 
\Bigg( \int_{d(z_i^\circ, z')| \geq R} \Big|
d(z_i^\circ,z')^{2-n_i} 
\exp(-ckd(z_i^\circ,z') \Big|^{p'}   \, dz' \Bigg)^{1/p'} \\
\leq  \begin{cases} C k^{n_i/p - 1}, \quad p > \frac{n_i}{2} \\ C k \log k, \quad p = \frac{n_i}{2} \\ Ck, \quad p < \frac{n_i}{2}. \end{cases}
\end{gathered}\label{k-calc}\end{equation}
 Clearly this is an integrable function of $k$ as $k \to 0$ for any $p \in (1, \infty)$. 

It follows that it suffices to consider the unboundedness of the operator \eqref{badterm3} when we replace $\nabla \phi_i(z) + \nabla u_i(z,k)$ by $\nabla (\phi_i(z) +u_i(z,0)) = \nabla \Phi_i$; recall  that $\Phi_i := \phi_i + u_i(\cdot, 0)$ is the unique harmonic function that tends to $1$ at infinity along the $i$th end and to zero along every other end (uniqueness is implied by the maximum principle). So it suffices to show that 
\begin{equation}
\tau(z)   \int_0^{k_0}  \nabla \Phi_i(z) 
d(z_i^\circ,z')^{2-n_i} 
\exp(-ckd(z_i^\circ,z')  \phi_i(z') \, dk
\label{badterm4}\end{equation}
does not act boundedly on $L^p(\M)$  for $p \geq \min_i n_i$. This is given along the $i$th end in the $z'$ coordinates by 
\begin{equation}
  \frac{\tau(z) \nabla \Phi_i(z)  }{d(z_i^\circ,z')^{n_i - 2}} \int_0^{k_0} \exp\left(-ckd(z_i^\circ,z')\right) \, dk =   \frac{\tau(z) \nabla \Phi_i(z)  }{d(z_i^\circ,z')^{n_i - 1}}  \big[1- \exp\left(-ck_0d(z_i^\circ,z')\right)\big]
\label{badterm5}\end{equation}
Notice that this is an operator of rank (at most) one.  Again, we need to check whether  it has the form $a(z) \ang{b(z), \cdot}$ with $a \in L^p$ and $b \in L^{p'}$. 

The function $b$ here is in $L^{p'}$ precisely for $p < n_i$. Noting that $\nabla \Phi$ cannot vanish on any open set, as $\Phi$ is harmonic and nonconstant\footnote{It is here that we use the hypothesis that there are at least two ends. If there is only one end, then $\Phi$ is constant, its gradient vanishes identically, and we have boundedness of the Riesz transform for all $p \in (1, \infty)$.}, we conclude that $a \neq 0$ in $L^p$ and therefore, this kernel does not act boundedly on $L^p(\M)$ for $p \geq n_i$. As this argument applies to each end, we see that we fail to have boundedness for $p \geq \min_i n_i$. 
\end{proof}

\section{A generalisation}\label{gen}

Our main result can be generalised to a setting wider than the Cartesian product of $\R^n\times \mathcal{M}_i$. 
Consider a family of $N$-dimensional Riemannian manifolds $\mathcal{V}_1, \mathcal{V}_2, \ldots, \mathcal{V}_l$ with bounded geometry and positive injectivity radius, each with a smooth nondegenerate measure $\mu_i$ (not necessarily the Riemannian measure). Assume that

\begin{equation}
\mu_i(B(z,r)) \le \begin{cases} C r^{N}, \quad r \le 1  \\ C r^{n_i}, \quad r > 1 \end{cases} \text{ for some } n_i \geq 3. 
\label{vol}\end{equation}
holds for all $i=1,\ldots, l$ and corresponding manifolds $\mathcal{V}_i$. 

Let  $\mathcal{V}_1, \dots, \mathcal{V}_l$ be a family of Riemannian manifolds as above. 
Let $ \Delta_{\mathcal{V}_i}$ be the Laplace-Beltrami operator on $\mathcal{V}_i$ defined by \eqref{LapV}. 
We assume that 
   \begin{equation}\label{grad}
\|\nabla \exp(-t\Delta_{\mathcal{V}_i})\|_{1\to \infty } \le
 \begin{cases}
 Ct^{-(N+1)/2} & t \leq 1\\
 Ct^{-(n_i+1)/2} & t>1. \end{cases} 
 \end{equation} 
Under the doubling condition \eqref{doub} it automatically follows that
   \begin{equation}\label{hkb}
  \|\exp(-t\Delta_{\mathcal{V}_i})\|_{1\to \infty } \le
  \begin{cases}
  Ct^{-N/2} & t \leq 1\\
  Ct^{-n_i/2} & t>1, \end{cases}
  \end{equation}
see \cite[Collorary 2.2]{CS1}.  
  
 \begin{theorem}\label{th8.2}
 Let $\mathcal{V}_1, \dots, \mathcal{V}_l$ be a family of Riemannian manifolds with
 boundary geometry and positive injectivity radius, satisfying the doubling condition \eqref{doub} and estimates
 	\eqref{vol}, \eqref{hkb} and \eqref{grad}. 
 	Suppose that $\M=\mathcal{V}_1 \# \ldots \# \mathcal{V}_l$.  
	Then the corresponding Riesz transform $\nabla \Delta_\M^{-1/2}$ 
 is of weak type $(1,1)$ and is bounded on $L^p(\M)$ if and only if
 $1<p < \min{n_i}$.
  \end{theorem}

 \begin{proof}
 We have written the  proof of Theorem \ref{thm:main} in such a way that it generalizes almost immediately
 to this more general situation. The low energy estimate for the Riesz kernel follows directly. In fact, the ingredients
 for the low energy proof are (i) resolvent estimates and gradient resolvent estimates, which were deduced from heat kernel estimates \eqref{hkb} and
 \eqref{grad}, (ii) the spectral multiplier result Lemma~\ref{l2.1}, which holds in this greater generality, (iii) Proposition~\ref{prop:CGT}, which holds in 
 this generality, and (iv) the boundedness of the Riesz transform on the individual spaces $\mathcal{V}_i$, which,  by \eqref{grad}, holds due to  \cite{ACDH}, see also \cite{CS1}. 
 
 The high energy estimate in 
Proposition \ref{prop:he-Riesz} also  remains valid under the assumptions of 
Theorem~\ref{th8.2}. In fact, the part of argument corresponding to the term 
$ G''_r(\sqrt{ \Delta})$ does not require any changes. 
It remains to consider the term $ G'_r(\sqrt{ \Delta})$. Here we need to replace the part of the argument that
involves the functional calculus for pseudodifferential operators. 
Using a similar approach as in our discussion for  $G''_r$ in proof of Proposition \ref{prop:he-Riesz}   
we first prove continuity of the operator  
$ dG'_r(\sqrt{ \Delta_0})$ for $1<p\le 2$. The standard argument shows that  
$$
\left(dG'_r(\sqrt{ \Delta_0})\right)^*=d^*G'_r(\sqrt{ \Delta_1})
$$
The continuity  of $ dG'_r(\sqrt{ \Delta_0})$ for $2<p<\infty$ is equivalent to 
boundedness of $d^*G'_r(\sqrt{ \Delta_1})$ again for $1<p\le 2$.
It ie clear that  both operators $ dG'_r(\sqrt{ \Delta_0})$ and $d^*G'_r(\sqrt{ \Delta_1})$ are bounded on $L^2(\M)$. So it is enough to prove that they are also
of weak type $(1,1)$. However instead of proving these operators satisfies the standard Calder\'on-Zygmund condition, we use the approach from  \cite{Si} which is a variant of the technique developed in \cite{CD}. It is not difficult to see that one can use the argument described in \cite{Si} to prove weak $(1,1)$ estimates for the operators $ dG'_r(\sqrt{ \Delta_0})$ and $d^*G'_r(\sqrt{ \Delta_1})$ 
using the fact that the kernels of these operators are supported only in part of $\M^2$ 
which is close to diagonal, $d(x,y) \le r$. Note that for balls in $\M$ with radius smaller than a fixed constant, the doubling condition still holds in our setting. Thus a localisation of the argument from \cite{Si} proves Proposition \ref{prop:he-Riesz}
in the present setting. 

This concludes the proof of the theorem. 
 	\end{proof}
 
 \begin{remark}
 Continuity of the  Riesz transform localised to high-energy part  $\nabla(I+\Delta)^{-1/2}$ 
 in the setting of Theorem \ref{th8.2}
 was proved by Bakry in \cite{Ba2, Ba3}, see also \cite{ACDH}.  Note that
 $$
 \nabla F_>(\sqrt{\Delta})=\nabla(I+\Delta)^{-1/2} 
 F_>(\sqrt{\Delta})(I+\Delta)^{1/2}
 $$
 so Proposition \ref{prop:he-Riesz} follows from Bakry's result if 
 the operator $F_>(\sqrt{\Delta})(I+\Delta)^{1/2}$ is bounded on $L^p(\M)$ 
 for all $1\le p \le \infty$. This is a statement similar to Lemma~\ref{l2.1}.
 However, the doubling condition fails in this setting so one cannot use 
 the lemma directly to prove $L^p$ boundedness of $F_>(\sqrt{\Delta})(I+\Delta)^{1/2}$. Nevertheless one can prove this using the results 
 obtained in \cite{CSY}. It is also possible to adjust arguments in \cite{Ba2, Ba3} to directly prove boundedness of the operator $\nabla F_>(\sqrt{\Delta})$.
 \end{remark}
 
 \begin{remark} We expect that Theorem~\ref{th8.2} holds provided that $n_i > 2$. However, there are some modifications that need to be made to the proof for $2 < n_i < 3$. For example
 the factor of $k$ in \eqref{u-kminus0-grad} has to be replaced by a weaker power $k^{n_i - 2}$, and then the calculation \eqref{k-calc} has to be modified. As we do not know of interesting  examples of manifolds with $2 < n_i < 3$ satisfying \eqref{hkb} and \eqref{grad}, we do not pursue this further here. 
 \end{remark}
 
 \begin{remark}
 One can generalise further to manifolds which do not necessarily have bounded geometry, but instead
 satisfy a one-sided condition on curvature. 
	
Recall the Hodge Laplacian acting on $1$-forms, 
$ \Delta_1 = d d^* + d^* d$, 
 can be expressed using the Bochner formula in the following 
 way:
 \begin{equation}\label{bochner}
 {\Delta}_{1}=\bar{\Delta}_1+\mathscr{R}-\mathscr{H}_f,
 \end{equation}
 where $\bar{\Delta}_1$ is  the weighted rough Laplacian, $\mathscr{R}$ is the curvature tensor and $\mathscr{H}_f$ is a Hessian which can be defined in as follows, see \cite{CDS}. Write the measure $\mu$ as $\mu = e^{f}dx$ where $dx$ is Riemannian measure.
 Then for  a smooth function $f$ on $\M$ the Hessian of $f$ is the symmetric $(0,2)$-tensor defined as
$$\mathrm{Hess}_f(X,Y)=\nabla_X\nabla_Yf-\nabla_{\nabla_XY}f.$$ 
Then the bounded geometry assumption may be weakened to the following condition:
There exists a constant $C_M$ such that 
 $$
 \mathscr{R}-\mathscr{H}_f \ge -C_m I
 $$
 as a quadratic form.

	\end{remark}
 
\noindent
{\bf Acknowledgements:}
Both authors were  partly  supported by
Australian Research Council  Discovery Grant DP DP160100941.


\begin{thebibliography}{99}


\bibitem{AS}
{M.~Abramowitz and I.A. Stegun,}
\newblock {\em  Handbook of mathematical functions with formulas, graphs, and
              mathematical tables,}  
\newblock {National Bureau of Standards Applied Mathematics Series}, {55},
U.S. Government Printing Office, Washington, D.C. (1964).

\bibitem{Au}
P. Auscher, 
\newblock On necessary and sufficient conditions for $L^p$-estimates of Riesz transforms associated to elliptic operators on $\Bbb R^n$ and related estimates, 
\newblock{\em  Mem. Amer. Math. Soc.}  186  (2007). 

\bibitem{ACDH} {P.~Auscher, T.~Coulhon, X.T.~Duong and S.~ Hofmann,
\newblock Riesz transform on manifolds and heat kernel regularity, \newblock 
\emph{Ann. Sc. E. N. S.}, 37:911--957, (2004). }



\bibitem{Ba2}
{D.~Bakry},
\newblock {\'{E}tude des transformations de {R}iesz dans les vari\'et\'es
              riemanniennes \`a courbure de {R}icci minor\'ee},
\newblock in {\em S\'eminaire de Probabilit\'es XIX} {Lecture Notes in Math.},{1247}:{137--172}, {Springer}, {Berlin}, {(1987)}.

\bibitem{Ba3}
{D.~Bakry},
\newblock {The Riesz transforms associated with second order differential operators. Seminar on Stochastic Processes,}
{88}:{1--43}, Birkhäuser, Basel, 1989.


\bibitem{BDLW} T.~A. Bui, X.~T. Duong, J. Li, B. Wick, 
\newblock Functional calculus of operators with generalised Gaussian bounds on non-doubling manifold with ends, 
\newblock arXiv:1804:11099v1. 


\bibitem{CZ}
A.P. Calder\'on and  A. Zygmund, 
\newblock On the existence of certain singular integrals,
\newblock {\em Acta Math.} 88: 85--139, (1952).

\bibitem{Ca2} G. Carron,
\newblock Riesz transforms on connected sums.
\newblock  {\em Ann. Inst. Fourier (Grenoble)} 57(7):2329--2343, (2007).



\bibitem{Ca1} G. Carron, 
\newblock Riesz transform on manifolds with quadratic curvature decay
\newblock {\em Rev. Mat. Iberoam} 33:749--788, (2017).




\bibitem{CCH}
G.~Carron, T.~Coulhon and  A.~Hassell, 
\newblock Riesz transform and {$L\sp p$}-cohomology for manifolds with
{E}uclidean ends
\newblock {\em Duke Math. J.  }, 133:59--93, (2006).

\bibitem{CGT}
J. Cheeger, M. Gromov and M. Taylor, 
\newblock Finite propagation speed, kernel estimates for functions of the Laplace operator, and the geometry of complete Riemannian manifolds, 
\newblock \emph{J. Diff. Geom.} \textbf{17}(1982), 15-53.


\bibitem{CSY}	P. Chen, A. Sikora, L. Yan,
\newblock Spectral multipliers via resolvent type estimates on non-homogeneous metric measure spaces, arXiv:1609.01871


\bibitem{CSi}
M.~Cowling and A.~Sikora,
\emph{A spectral multiplier theorem for a sublaplacian on $\mathrm{SU}(2)$},
{Math.  Z.} {238}:1--36, (2001).





\bibitem{CoW} R.R. Coifman, G. Weiss,
\emph{Analyse harmonique non-commutative sur certains espaces
homog\`enes}, Lecture Notes in Mathematics, Vol. 242.
Springer-Verlag, Berlin--New York, (1971).



\bibitem{CDS} 
 T. Coulhon, B. Devyver and A.  Sikora, 
 \newblock Gaussian heat kernels: from functions to forms, 
 \newblock {\em to appear at J. Reine Angew. Math.} 
arXiv:1606.02423, (2016).

\bibitem{CD}
T. Coulhon and X.~T. Duong,
\newblock Riesz transforms for $1\leq p\leq 2$,
\newblock {\em Trans. Amer. Math. Soc.}, 351(3):1151--1169, (1999).

\bibitem{CMZ}
T. Coulhon, D. M\"uller and  J. Zienkiewicz, 
\newblock About Riesz transforms on the Heisenberg groups,
\newblock {\em Math. Ann.}, 305(2): 369--379, (1996). 

\bibitem{CS1} T. Coulhon, A. Sikora, Riesz meets Sobolev, {\em Colloq. Math.}, 118,  685-704, 2010.

%
%
\bibitem{Dev}
B. Devyver,
\newblock A perturbation result for the Riesz transform,
\newblock {\em Ann. Sc. Norm. Super. Pisa Cl. Sci.}, XIV (5):937--964, (2015). 

\bibitem{Dev2}  B. Devyver,
\newblock  On gradient estimates for the heat kernel, arXiv:1803.10015v3. 

\bibitem{Dunau}
J. Dunau, 
\newblock Fonctions d'un op\'erateur elliptique sur une vari\'et\'e  compacte, 
\newblock \emph{J. Math. Pures Appl.} 56:367--391, (1977). 

\bibitem{DOS}
X.~T. Duong, E.-M. Ouhabaz  and A. Sikora. \emph{Plancherel-type
	estimates and sharp spectral multipliers}, 
J. Funct.\ Anal.  \textbf{196} (2002), 443--485. 

\bibitem{ERS} A.F.M. ter Elst, D.W. Robinson, A.
Sikora, On second-order periodic
elliptic operators in divergence form. {\em Math. Z.} 
238  (2001)   no. 3, pp. 569--637.
 
 \bibitem{GS99} A. Grigor'yan and L. Saloff-Coste, {Heat kernel on connected sums of Riemannian manifolds},
\emph{Math. Res. Lett.} 6  (1999), no. 3-4, 307--321.

 
\bibitem{GS} A. Grigor'yan and L. Saloff-Coste, {Heat kernel on manifolds with ends},
{\em Ann. Inst. Fourier (Grenoble)}  {\bf 59} (2009), no.5, 1917--1997.

\bibitem{GS2} A. Grigor'yan and L. Saloff-Coste, {Surgery of the Faber-Krahn inequality and applications to heat kernel bounds},
{\emph{Nonlinear Analysis}}  {\bf 131} (2016), 243--272.


\bibitem{GH1} C. Guillarmou, A. Hassell, 
\emph{Resolvent at low energy and Riesz transform for Schrodinger operators on asymptotically conic manifolds, I }.
Math. Ann. \textbf{341} (2008), 859--896.

\bibitem{GH2} C. Guillarmou, A. Hassell, 
\emph{Resolvent at low energy and Riesz transform for Schrodinger operators on asymptotically conic manifolds, II}, Annales de l'Institut Fourier \text{59} (2009), 1553 -- 1610.

%

%

\bibitem{GHS} C. Guillarmou, A. Hassell, A. Sikora, 
\newblock{Restriction and spectral multiplier theorems on asymptotically conic manifolds}, 
\newblock {\em Anal. PDE},
{6}:893--950 (2013).

\bibitem{HS}
A.~Hassell and  A.~Sikora,
\newblock Riesz transforms in one dimension,
\newblock {\em Indiana Univ. Math. J.} 58(2):823--852 (2009).

\bibitem{HV}
A.~Hassell and  A.~Vasy,
\newblock Symbolic functional calculus and $N$-body resolvent estimates,
\newblock \emph{J.  Funct. Anal.} 173:257--283 (2000). 


\bibitem{HeSt} W. Hebisch, T. Steger, 
\newblock Multipliers and singular integrals on exponential growth groups, 
{\em Math. Z.} 245(1):37--61 (2003). 


\bibitem{Ji1}  R. Jiang,
\newblock  Riesz transform via heat kernel and harmonic functions on non-compact manifolds, arXiv:1710.00518v3. 

\bibitem{Ji2}  R. Jiang, F. Lin
\newblock  Riesz transform under perturbations via heat kernel regularity, arXiv:1808.01948v1.
. 

\bibitem{Lohoue}
N. Lohou\'{e}, 
\newblock Comparaison des champs de vecteurs et de puissances du Laplacian sur une vari\'{e}t\'{e} \`{a} courbure non positive, 
\newblock \emph{J. Funct. Anal.} 61 (1985), 164-201.


\bibitem{NTV}
F. Nazarov, S. Treil,  A. Volberg, 
\newblock Weak type estimates and Cotlar inequalities for Calder—n-Zygmund operators on nonhomogeneous spaces. 
\newblock {\em Internat. Math. Res. Notices} 9:463--487, 1998. 

\bibitem{Ri} 
M. Riesz, 
\newblock Sur les fonctions conjugu\'es, 
\newblock {\em Math. Zeitschrift,} 27:218--244 (1928).

%

\bibitem{Sal1}  L. Saloff-Coste, Analyse sur les groupes de Lie \`a croissance polyn\^{o}miale, {\em Ark. Mat.},  28 (1990),  315-331.




\bibitem{Si} A.~Sikora, \newblock Riesz transform, Gaussian bounds
and the method of wave equation, {\em Math. Z.}, 247(3):643--662, (2004). 

\bibitem{St} E.M.~Stein,
\newblock {Some results in harmonic analysis in {${\bf R}^{n}$}, for
{$n\rightarrow \infty $}}
{\em Bull. Amer. Math. Soc. (N.S.)}, 9(1):71--73, (1983). 

\bibitem{Strichartz} R. Strichartz, 
\newblock Analysis of the Laplacian on the complete Riemannian manifold,
\newblock \emph{J. Funct. Anal.} 52:48--79 (1983). 

\bibitem{Taylor} M.E.~Taylor, 
\newblock Pseudodifferential operators,
\emph{Princeton University Press}, 1981. 

\bibitem{To1}
X. Tolsa,
\newblock Calder—n-Zygmund theory with non doubling measures. 
Nonlinear analysis, function spaces and applications. 
{\em Acad. Sci. Czech Repub. Inst. Math.},  9:217--260,   Prague, (2011). 

\bibitem{To2}
X. Tolsa, 
\newblock Littlewood-Paley theory and the T(1) theorem with non-doubling measures. 
{\em Adv. Math. }164(1), 57--116, (2001). 

\bibitem{Tr} \textrm{C. J.~Tranter,
	\newblock {\em Bessel functions with some physical applications,} \newblock Hart Publishing Co. Inc., New York, (1969). }

%
%




\end{thebibliography}
\end{document}